\documentclass[11]{amsart}
\usepackage{amsaddr}

\newcommand{\Ell}{\tfs{L}}

\newcommand{\cont}{\tfs{C}}

\newcommand{\Hom}{{\tfs{Hom}}}
\newcommand{\Ker}{\tfs{Ker}\,}

\newcommand{\id}{\tfs{id}}
\newcommand{\po}{{\tfs{po}}}
\newcommand{\funlattice}{{\mathcal F}}

\newcommand{\zerospace}{\{0\}}

\newcommand{\Cstar}{\ensuremath{\tfs{C}^\ast}}
\newcommand{\Calgebra}{\Cstar\!-algebra}

\newcommand{\homo}{{\mathcalalt H}}
\newcommand{\homon}{\homo_n}
\newcommand{\incr}{{\tfs{i}}}
\newcommand{\funincrn}{{\mathcalalt F}_n^\incr}
\newcommand{\ci}{{\mathcalalt I}}
\newcommand{\cia}{\ci^{\tfs{a}}}
\newcommand{\cic}{\ci^{\tfs{c}}}

\newcommand{\ct}{{\mathcalalt C}}
\newcommand{\ccond}{{\mathcalalt D}}

\newcommand{\ol}[1]{\overline{#1}}
\newcommand{\fac}[1]{\overline{#1}}

\newcommand{\Fun}{\tfs{Fun}\hskip 0.07em}

\newcommand{\fact}[1]{{\overline{#1}}}

\newcommand{\pospartofset}[1]{#1_+}
\newcommand{\normdualofbs}[1]{#1^\ast}
\newcommand{\algdualofvs}[1]{#1^\sharp}

\newcommand{\vs}{V}
\newcommand{\vstwo}{W}
\newcommand{\ls}{L}
\newcommand{\vspos}{\pospartofset{\vs}}
\newcommand{\povs}{(\vs,\vspos)}
\newcommand{\povsclosure}{(\vs,\overline{\vspos}^{\sigma(\vs,\lsdual)})}
\newcommand{\vsdual}{\algdualofvs{\vs}}

\newcommand{\lsdual}{\algdualofvs{\ls}}
\newcommand{\lsdualpos}{\pospartofset{\lsdual}}

\newcommand{\vl}{E}
\newcommand{\vltwo}{F}
\newcommand{\vlpos}{\pospartofset{\vl}}

\newcommand{\bl}{E}
\newcommand{\blpos}{\pospartofset{\bl}}

\newcommand{\bldual}{\normdualofbs{\bl}}

\newcommand{\bltwo}{F}

\newcommand{\ball}{{\tfs{B}}}
\newcommand{\ballbs}{\ball_{\bs}}
\newcommand{\ballbsdual}{\ball_{\bsdual}}
\newcommand{\ballbsdualpos}{\left(\ball_{\bsdual}\right)_+}	

\newcommand{\bs}{X}
\newcommand{\bspos}{\pospartofset{\bs}}
\newcommand{\bsdual}{\normdualofbs{\bs}}
\newcommand{\bsdualpos}{\bsdual_+}

\newcommand{\pobs}{(\bs,\bspos)}
\newcommand{\pobsclosure}{(\bs,\overline{\bspos})}

\newcommand{\FVL}{{\tfs{FVL}}}
\newcommand{\FVLpo}[1]{\FVL[(#1,\pospartofset{#1})]}
\newcommand{\FVLpovs}{\FVLpo{\vs}}
\newcommand{\FVLpobs}{\FVLpo{\bs}}
\newcommand{\FVLpovsclosure}{\FVL[(\vs,\overline{\vspos}^{\sigma(\vs,\lsdual)})]}
\newcommand{\FVLpobsclosure}{\FVL[(\bs,\overline{\bspos})]}	
\newcommand{\FBL}{{\tfs{FBL}}}
\newcommand{\FBLC}{\FBL^\ct}
\newcommand{\FBLCpo}[1]{\FBLC[(#1,\pospartofset{#1})]}
\newcommand{\FBLCpobs}{\FBLCpo{\bs}}
\newcommand{\FBLp}{\FBL^{(p)}}
\newcommand{\FBLppo}[1]{\FBLp[(#1,\pospartofset{#1})]}
\newcommand{\FBLppobs}{\FBLppo{\bs}}
\newcommand{\FBLinf}{\FBL^{(\infty)}}
\newcommand{\FBLinfpo}[1]{\FBLinf[(#1,\pospartofset{#1})]}
\newcommand{\FBLinfpobs}{\FBLinfpo{\bs}}

\newcommand{\conth}{\cont_{\tfs {ph}}}
\newcommand{\conthp}{\conth\sp p}
\newcommand{\contballposh}{\conth(\ballbsdualpos)}
\newcommand{\contballposhp}{\conthp(\ballbsdualpos)}

\newcommand{\freeelt}{\mathfrak f}

\newcommand{\mor}{\varphi}
\newcommand{\mortwo}{\psi}

\newcommand{\facmor}{\ol{\mor}}
\newcommand{\functor}{U}

\newcommand{\catfont}[1]{\textup{\fontfamily{qag}\selectfont{\footnotesize #1}}}
\newcommand{\catfontscript}[1]{\textup{\fontfamily{qag}\selectfont{\tiny{#1}}}}

\newcommand{\catsymbol}{\catfont{Cat}}
\newcommand{\catone}{\catsymbol_1}
\newcommand{\cattwo}{\catsymbol_2}

\newcommand{\catsymbolscript}{\catfontscript{Cat}}
\newcommand{\catonescript}{{\catsymbolscript}_1}
\newcommand{\cattwoscript}{{\catsymbolscript}_2}

\newcommand{\pset}{\Omega}
\newcommand{\pt}{\omega}

\newcommand{\freeletter}{\tfs{F}}
\newcommand{\Free}[3]{\freeletter_{#1}^{#2}\!\left[#3\right]}
\newcommand{\FCATONECATTWO}[1]{\Free{\catonescript}{\cattwoscript}{#1}}

\usepackage[english]{babel} 
\usepackage[UKenglish]{datetime}

\usepackage{amscd}
\usepackage{amsfonts}
\usepackage{amsmath}
\usepackage{amssymb}
\usepackage{amsthm}
\usepackage{thmtools} 
\usepackage{cases}
\usepackage{chngcntr}
\usepackage{cite}
\usepackage{color}
\usepackage{graphicx}
\usepackage{mathtools}
\usepackage{microtype}
\usepackage[active]{srcltx}
\usepackage{tikz}
\usepackage{tikz-cd}
\usepackage{url}
\usepackage{verbatim} 

\usepackage{tikz,tikz-cd}
\usetikzlibrary{matrix, arrows}

\usepackage{hyperref}
\usepackage{enumitem}
\usepackage{cleveref}

\newcommand*\patchAmsMathEnvironmentForLineno[1]{%
  \expandafter\let\csname old#1\expandafter\endcsname\csname #1\endcsname
  \expandafter\let\csname oldend#1\expandafter\endcsname\csname end#1\endcsname
  \renewenvironment{#1}%
     {\linenomath\csname old#1\endcsname}%
     {\csname oldend#1\endcsname\endlinenomath}}%
\newcommand*\patchBothAmsMathEnvironmentsForLineno[1]{%
  \patchAmsMathEnvironmentForLineno{#1}%
  \patchAmsMathEnvironmentForLineno{#1*}}%
\AtBeginDocument{%
\patchBothAmsMathEnvironmentsForLineno{equation}%
\patchBothAmsMathEnvironmentsForLineno{align}%
\patchBothAmsMathEnvironmentsForLineno{flalign}%
\patchBothAmsMathEnvironmentsForLineno{alignat}%
\patchBothAmsMathEnvironmentsForLineno{gather}%
\patchBothAmsMathEnvironmentsForLineno{multline}%
}

\usepackage[mathlines]{lineno}

\usepackage[mathscr]{euscript}
\usepackage{calrsfs}
\DeclareMathAlphabet{\mathcalalt}{OMS}{cmsy}{m}{n}
\usepackage[T1]{fontenc}
\usepackage[charter]{mathdesign}
\usepackage{bbm} 
\newcommand{\bbfont}{\mathbbm}

\usepackage{calrsfs}

\DeclareMathAlphabet{\mathcalalt}{OMS}{cmsy}{m}{n}
\DeclareMathAlphabet{\mathcalaltbf}{OMS}{cmsy}{b}{n}

\usepackage[mathscr]{euscript}

\usepackage{ifthen}

\makeatletter
\newcommand{\tfs}[1]
{
	\ifthenelse{\equal{\f@shape}{n}}{\ensuremath{\mathrm{#1}}}
	{\ifthenelse{\equal{\f@shape}{sc}}{\ensuremath{\mathrm{#1}}}
		{\ifthenelse{\equal{\f@shape}{it}}{\ensuremath{\mathit{#1}}}
			{\ifthenelse{\equal{\f@shape}{sl}}{\ensuremath{\mathit{#1}}}{}	
			}
		}
	}
}
\makeatother

\makeatletter
\newcommand{\btfs}[1]
{
	\ifthenelse{\equal{\f@shape}{n}}{\ensuremath{\mathrm{#1}}}
	{\ifthenelse{\equal{\f@shape}{sc}}{\ensuremath{\mathrm{#1}}}
		{\ifthenelse{\equal{\f@shape}{it}}{\ensuremath{\mathit{#1}}}
			{\ifthenelse{\equal{\f@shape}{sl}}{\ensuremath{\mathit{#1}}}{}	
			}
		}
	}
}
\makeatother

\newcommand{\NN}{{\bbfont N}}

\newcommand{\RR}{{\bbfont R}}

\newcommand{\ulpar}{{\textup{(}}}
\newcommand{\urpar}{{\textup{)}}}

\newcommand{\uppars}[1]{\ulpar #1\urpar}

\newcommand{\abs}[1]{{\lvert #1 \rvert}}
\newcommand{\norm}[1]{{\lVert #1 \rVert}}

\newcommand{\lrabs}[1]{{\left\lvert #1 \right\rvert}}
\newcommand{\lrnorm}[1]{{\left\lVert #1 \right\rVert}}

\newcommand{\Bignorm}[1]{{\Big\lVert #1 \Big\rVert}}

\theoremstyle{plain}

\newtheorem{theorem}{Theorem}[section]
\newtheorem{proposition}[theorem]{Proposition}
\newtheorem{lemma}[theorem]{Lemma}
\newtheorem{corollary}[theorem]{Corollary}

\newtheorem*{theorem*}{Theorem}
\newtheorem*{proposition*}{Proposition}
\newtheorem*{lemma*}{Lemma}
\newtheorem*{corollary*}{Corollary}

\theoremstyle{definition}

\newtheorem{definition}[theorem]{Definition}
\newtheorem{example}[theorem]{Example}
\newtheorem{remark}[theorem]{Remark}

\newtheorem*{definition*}{Definition}
\newtheorem*{example*}{Example}
\newtheorem*{remark*}{Remark}

\crefname{theorem}{Theorem}{Theorems}
\crefname{proposition}{Proposition}{Propositions}
\crefname{lemma}{Lemma}{Lemmas}
\crefname{corollary}{Corollary}{Corollaries}
\crefname{definition}{Definition}{Definitions}
\crefname{example}{Example}{Examples}
\crefname{remark}{Remark}{Remarks}

\crefname{section}{Section}{Sections}
\crefname{subsection}{Section}{Sections}
\crefname{subsubsection}{Section}{Sections}

\crefname{equation}{equation}{equations}
\crefname{enumi}{part}{parts}
\crefname{enumii}{part}{parts}
\crefname{enumiii}{part}{parts}
\crefname{enumiv}{part}{parts}

\numberwithin{equation}{section}

\allowdisplaybreaks

\usepackage[nobysame]{amsrefs}

\AtBeginDocument{\def\MR#1{}}

\begin{document}


\title [Free Banach lattices]{Free Banach lattices over pre-ordered Banach spaces}

\author{Marcel de Jeu}
\address[Marcel de Jeu]{Mathematical Institute, Leiden University, P.O.\ Box 9512, 2300 RA Leiden, The Netherlands\\
	and\\
	Department of Mathematics and Applied Mathematics, University of Pretoria, Corner of Lynnwood Road and Roper Street, Hatfield 0083, Pretoria,
	South Africa}
\email[Marcel de Jeu]{mdejeu@math.leidenuniv.nl}

\author{Xingni Jiang}
\address[Xingni Jiang]{College of Mathematics, Sichuan University, No.\ 24, South Section, First Ring Road, Chengdu, P.R.\ China}
\email[Xingni Jiang]{x.jiang@scu.edu.cn}

\subjclass[2010]{Primary 46B42, 46A40; Secondary 06F20, 47B60}

\keywords{Pre-ordered vector space, pre-ordered Banach space, free vector lattice, free Banach lattice, convexity}

\begin{abstract}We study free Banach lattices over pre-ordered Banach spaces in the category of Banach lattices of a given convexity type. These generalise the free Banach lattices under convexity conditions over Banach spaces in the literature. Their existence is shown from the existence of free vector lattices over pre-ordered vector spaces, which are also investigated.
\newline\indent We determine when the positive contraction from the pre-ordered Banach space into the free Banach lattice is injective or bipositive, and when it has closed range. It is a bipositive embedding with closed range if and only if the positive wedge of the space is a closed normal cone. Even for a Banach lattice it can be non-isometric.
\newline\indent 
By analysing the norm of the free $p$-convex Banach lattice with convexity constant 1 over a pre-ordered Banach space, it becomes clear that it can be realised as a function lattice on the positive part of the dual unit ball. This generalises the known realisation for a free Banach lattice of that type over a Banach space.
\newline\indent
As a preparation for this analysis of the norm, characterisations of $p$-convex Banach lattices in terms of vector lattice homomorphisms into $\mathrm{L}_p(\mu)$-spaces for probability measures $\mu$ are given.
\end{abstract}

\maketitle


\section{Introduction and overview}\label{sec:introduction_and_overview}

\noindent The notion of the free Banach lattice over a set was introduced by de Pagter and Wickstead in \cite{de_pagter_wickstead:2015}. It was followed by the definition of the free Banach lattice $\FBL[\bs]$ over a Banach space $\bs$ by Avil\'es, Rodr\'{\i}guez, and Tradacete in \cite{aviles_rodriguez_tradacete:2018};  simplifications in its construction were pointed out by Troitsky in \cite{troitsky:2019}. The seminal paper \cite{aviles_rodriguez_tradacete:2018} was the starting point of an active line of research. The existence and properties of $\FBL[\bs]$ are not only of interest in itself, but the isometric embedding of $\bs$ into it also makes a better understanding possible of the relation between Banach spaces and Banach lattices.  For a comprehensive treatment and additional references, we refer to \cite{oikhberg_taylor_tradacete_troitsky:2024}.

In the present paper, we introduce pre-ordering to the field. For a pre-ordered Banach space $\bs$ with (not necessarily closed) positive wedge $\bspos$ and a given convexity type $\ct$ (a notion introduced in the paper), we define and study the free Banach lattice $\FBLCpobs$ over $\pobs$.  When taking $\ct=\emptyset$, our $\FBL^{\emptyset}[(\bs,\zerospace)]$ is $\FBL[X]$ from \cite{aviles_rodriguez_tradacete:2018}; the free Banach lattices $\FBL^{\ccond}[X]$ under convexity conditions from \cite{jardon-sanchez_laustsen_taylor_tradacete_troitsky:2022} all occur as an $\FBLC[(\bs,\zerospace]$ in the present paper.

A part of the definition of $\FBLCpobs$ consists of a positive contraction $j\colon\bs\to\FBLCpobs$. It need not be injective (it can be the zero map into the zero lattice) or bipositive, nor need it have closed range. When $\bspos$ is a closed normal cone, however, it is a bipositive injection with closed range; the converse is also true. This embedding opens the way for a better understanding of the relation between such ordered Banach spaces and Banach lattices. It is not true that $j$ is always isometric when $\bspos$ is a closed normal cone. It is isometric when $\bspos=\zerospace$ (as in \cites{aviles_rodriguez_tradacete:2018,jardon-sanchez_laustsen_taylor_tradacete_troitsky:2022}), and also when $\bs$ itself is what we shall call a $\ct$-convex Banach lattice. However, even for $\ell_1^2$ it can already fail to be isometric.\footnote{See \cref{ex:not_an_isometry}.}

This paper is organised as follows.

\cref{sec:preliminaries} contains the necessary preliminaries.

\cref{sec:convexity} is concerned with convexity in Banach lattices. In its first part, a characterisation of $p$-convex Banach lattices is obtained in terms of vector lattice homomorphisms into $\Ell_p(\mu)$-spaces for probability measures $\mu$. This is instrumental to \cref{sec:realisation}. In its second part, the notion of a convexity type is introduced, generalising that of a convexity condition in \cite{jardon-sanchez_laustsen_taylor_tradacete_troitsky:2022}. If, for a Banach lattice, there is a norming family of contractive vector lattice homomorphisms into Banach lattices that are all of the same convexity type, then so is the domain lattice.

\cref{sec:free_vector_lattices_over_pre-ordered_vector_spaces} introduces the free vector lattice $\FVLpovs$ over a pre-ordered vector space $\vs$ with positive wedge $\vspos$. It is shown that it exists, and that it can canonically be realised as a function lattice on the positive part of any separating linear subspace $L^\sharp$ of the dual of $\vs$ such that $\vspos$ is $\sigma(V,L^\sharp)$-closed. This is then applied to a pre-ordered Banach space; it also yields an extension of the lattice-linear functional calculus.

In \cref{sec:free_banach_lattices_over_pre-ordered_banach_spaces}, we give a definition, as stipulated by category theory, of the free Banach lattice $\FBLCpobs$ of a given convexity type $\ct$ over a pre-ordered Banach space $\pobs$. With the existence of the free vector lattice over the pre-ordered vector space $\pobs$ as a starting point, it is not too difficult to show that it exists. This differs from the constructions in \cites{jardon-sanchez_laustsen_taylor_tradacete_troitsky:2022,oikhberg_taylor_tradacete_troitsky:2024,troitsky:2019}, where the starting point is a vector lattice of functions on the dual unit ball. Our approach simplifies the notation, and it also appears to be more in line with similar constructions elsewhere. In \cite[Proposition~2.7.1]{dixmier_C-STAR-ALGEBRAS_ENGLISH_NORTH_HOLLAND_EDITION:1977}), for example, the enveloping \Calgebra\ of an involutive Banach algebra is similarly constructed, by first taking a universal object in an algebraic category, and then using the universal property of the sought analytic object to `dress this up'. The same is done in \cite{van_amstel_THESIS:2023} when constructing Banach lattice algebras with certain universal properties. The starting point is then a free vector lattice algebra. There is no concrete model of this available{\textemdash}there appears to be little hope for this{\textemdash}but this is also not necessary for the construction to proceed.

In the construction of the free Banach lattices under convexity conditions in \cite{jardon-sanchez_laustsen_taylor_tradacete_troitsky:2022}, an investigation is made, building on \cite{lausten_troitsky:2020}, of the positively homogeneous functional calculus and its continuity properties for suitable, not necessarily uniformly complete, vector lattices. As our construction shows, this can be circumvented.

After investigating basic properties of $\FBLCpobs$, and notably those of the positive contraction from $\bs$ into it, \cref{sec:free_banach_lattices_over_pre-ordered_banach_spaces} concludes with a comparison of the definitions of free Banach lattices in \cite{aviles_rodriguez_tradacete:2018} and in the present paper. It is shown that they coincide when $\bspos=\zerospace$ and $\ct=\emptyset$, but also that the obvious adaptation of the definition in \cite{aviles_rodriguez_tradacete:2018} to the general case is not the correct one to work with.

In \cref{sec:realisation}, it is shown that the free $p$-convex Banach lattice with $p$-convexity constant 1 over $\pobs$ can be realised as a vector lattice of positively homogeneous continuous functions on the positive part of the dual unit ball. When $\bspos=\zerospace$, one retrieves the realisations in \cites{aviles_rodriguez_tradacete:2018, jardon-sanchez_laustsen_taylor_tradacete_troitsky:2022}. Our approach is fundamentally different from that in \cites{aviles_rodriguez_tradacete:2018, jardon-sanchez_laustsen_taylor_tradacete_troitsky:2022}, where the proof of the existence of the free Banach lattice is given in \cite{aviles_rodriguez_tradacete:2018}, and redone in \cite{jardon-sanchez_laustsen_taylor_tradacete_troitsky:2022}, by `simply' showing that the proposed function lattice has the required properties. We start from the free Banach lattice, which we know to exist, and analyse its norm to the extent where it becomes obvious that it has a realisation as a function lattice on the positive part of the dual unit ball. One could say that in the present paper these realisations are \emph{derived}, and that in \cites{aviles_rodriguez_tradacete:2018, jardon-sanchez_laustsen_taylor_tradacete_troitsky:2022} they are \emph{found}.

We conclude with two remarks on notation. Firstly, compared to many papers on free Banach lattices, our canonical choice of letters is reversed: $\bs$ is a Banach space and $\bl$ and $\bltwo$ are Banach lattices or vector lattices. We believe that this reflects the notation in most of the existing textbooks on Banach lattices and vector lattices, and apologise for any inconvenience this may cause. Secondly, we have consistently used Greek letters for positive contractions and contractive vector lattice homomorphisms, rather than the usual operator letters. This fits better into our categorical approach, where these maps are the morphisms we work with.

\section{preliminaries}\label{sec:preliminaries}

\noindent In this section, we collect the necessary notations, definitions, and conventions.

All vector spaces are over the real numbers. When $\vs$ is a vector space, we let  $\vs^\sharp$ denotes its dual of all linear maps from $\vs$ to $\RR$.

A vector lattice need not be Archimedean. If $S$ is a non-empty set, then $\Fun(S)$ is the vector lattice of all real-valued functions on $S$.

If $\vs$ is a vector space, then a \emph{wedge} in $\vs$ is a non-empty subset $\vspos$ such that $\alpha_1x_1+\alpha_2x_2\in \vspos$ for all $\alpha_1,\alpha_2\geq 0$ and $x_1,x_2\in \vspos$.
A wedge $\vspos$ is a \emph{cone} when $\vspos\cap (-\vspos)=\zerospace$. For $x,y\in\vs$, we write $x\leq y$ or $y\geq x$ when $y-x\in\vspos$. As is well known, this gives a correspondence between pre-orderings on $\vs$ and wedges in $\vs$, where the partial orderings correspond to the cones. A \emph{pre-ordered vector space} is a pair $\povs$, where $\vspos$ is a wedge in the vector space $\vs$. If $\povs$ and  $(\vstwo,\pospartofset{\vstwo})$ are pre-ordered vector spaces, then a linear map $\mor\colon\vs\to\vstwo$ is \emph{positive} when $\mor(\vspos)\subseteq\pospartofset{\vstwo}$; when $\vspos=\zerospace$, these are just the linear maps. If, for $x\in\vs$, $\mor(x)\in\pospartofset{\vstwo}$ if and only if $x\in\vspos$, then $\mor$ is \emph{bipositive}. For a non-empty subset $S^\sharp$ of $\vs^\sharp$, we let $\pospartofset{S^\sharp}$ denote the positive elements of $S^\sharp$. If $\vspos=\zerospace$, then $\pospartofset{S^\sharp}=S^\sharp$.

If $\bs$ is a Banach space, then we let $\ballbs$ denote its unit ball and $\bs^*$ its dual. A \emph{pre-ordered Banach space} is a pair $\pobs$, where $\bspos$ is a (not necessarily closed) wedge in the Banach space $\bs$.
The positive contractive functionals on $\bs$, $\ballbsdualpos$, will play an important part below.
When $\bspos=\zerospace$, $\ballbsdualpos=\ballbsdual$.

We supply $\ballbsdualpos$ with the relative weak*-topology, making it into a compact Hausdorff space, and let $\contballposh$ denote the continuous, positively homogeneous, real-valued functions on it.\footnote{The positive homogeneity of a function $f$ on $\ballbsdualpos$  is to be interpreted as $f(tx^\ast)=tf(x^\ast)$ whenever $t\geq 0$ and $x^\ast, tx^\ast\in\ballbsdualpos$.}
For $x\in\bs$, we define $\delta_x\in\contballposh$ by setting
\begin{equation}\label{eq:delta_definition}
\delta_x(x^*)\coloneqq x^*(x)
\end{equation}
for $x^*\in\ballbsdualpos$. If $x\in\bspos$, then $\delta_x\geq 0$. When $\bspos=\zerospace$, $\delta_x$ is defined on $\ballbsdual$.

To conclude, we recall the following definition of free objects from  \cite[Definition~8.22]{adamek_herrlich_strecker_ABSTRACT_AND_CONCRETE_CATEGORIES_THE_JOY_OF_CATS:2006}.

\begin{definition}\label{def:free_object}
	Suppose that $\catone$ and $\cattwo$ are categories, and that $\functor\colon\cattwo\mapsto\catone$ is a faithful functor.\footnote{Recall that $\functor$ is \emph{faithful} when the associated map $\functor\colon\Hom_{\cattwoscript}(O_2,O_2^\prime)\to\Hom_{\catonescript}(\functor(O_2),\functor(O_2^\prime))$ is injective for all objects $O_2,O_2^\prime$ of $\cattwo$. This will always be the case in the present paper.}. Take an object $O_1$ of $\catone$. A \emph{free object over $O_1$ of $\cattwo$ with respect to $\functor$} is a pair $(j,\FCATONECATTWO{O_1})$, where $\FCATONECATTWO{O_1}$ is an object of $\cattwo$ and $j:O_1\to\functor\left(\FCATONECATTWO{O_1}\right)$ is a morphism of $\catone$,
	with the property that, for every object $O_2$ of $\cattwo$ and every morphism $\mor:O_1\to \functor(O_2)$ of $\catone$, there exists a unique morphism  $\facmor: \FCATONECATTWO{O_1}\to O_2$ of $\cattwo$ such that the diagram
	\begin{equation*}\label{dia:free_object}
		\begin{tikzcd}
			O_1\arrow[r, "j"]\arrow[dr, "\mor", swap]& \functor\left(\FCATONECATTWO{O_1}\right)\arrow[d, "\functor(\ol{\mor})"]
			\\ & \functor(O_2)
		\end{tikzcd}
	\end{equation*}
	in $\catone$ is commutative.
\end{definition}

A free object $(j,\FCATONECATTWO{O_1})$ as in \cref{def:free_object} need not exist. However, if it exists, and if  $(j^\prime,\FCATONECATTWO{O_1}^\prime)$ is another such pair, then a standard argument shows that the unique morphism  $\ol{j^\prime}:\FCATONECATTWO{O_1}\to\FCATONECATTWO{O_1}^\prime$ of $\cattwo$ such that $\functor(\ol{j^\prime})\circ j=j^\prime$ is, in fact, an isomorphism with as its inverse the unique morphism  $\ol{j}:\FCATONECATTWO{O_1}^\prime\to\FCATONECATTWO{O_1}$ of $\cattwo$ such that $\functor(\ol{j})\circ j^\prime=j$. In particular, a free object over $O_1$ of $\cattwo$ with respect to $\functor$, if it exists, is determined up to an isomorphism of $\cattwo$.  In the case of existence we shall often simply speak of `the' free object $\FCATONECATTWO{O_1}$ over $O_1$ of $\cattwo$, the accompanying morphism $j$ being tacitly understood from the context.

In the current paper, we are concerned with two contexts for \cref{def:free_object}, an algebraic and an analytic one. The results in the algebraic context are the stepping stone for those in the analytic context.

\section{Convexity}\label{sec:convexity}

\noindent In this section, we collect the material on convexity that will be needed in the sequel. For the positively homogeneous functional calculus for Banach lattices that is used in it, we refer to \cite[Section~1.d]{lindenstrauss_tzafriri_CLASSICAL_BANACH_SPACES_VOLUME_II_FUNCTION_SPACES:1979}. It will be used repeatedly that it is compatible with vector lattice homomorphisms; this well-known fact follows from the uniqueness statement in  \cite[Theorem~1.d.1]{lindenstrauss_tzafriri_CLASSICAL_BANACH_SPACES_VOLUME_II_FUNCTION_SPACES:1979}.

\subsection{p-convexity} \cref{sec:realisation} is concerned with free $p$-convex Banach lattices over pre-ordered Banach spaces. As a preparation for this, we recall the relevant definitions from \cite[Section~1.d]{lindenstrauss_tzafriri_CLASSICAL_BANACH_SPACES_VOLUME_II_FUNCTION_SPACES:1979}, and give new  characterisations of $p$-convex Banach lattices in terms of contractive vector lattice homomorphisms into $\Ell_p(\mu)$-spaces for probability measures $\mu$.

Let $M\geq 1$. For $1\leq p<\infty$, a Banach lattice $\bl$ is said to be \emph{$p$-convex with $p$-convexity constant at most $M$} when
\[
\lrnorm{\left(\sum_{i=1}^n\abs{x_i}^p\right)^{1/p}}\leq M \left(\sum_{i=1}^n\norm{x_i}^p\right)^{1/p}
\]
for $x_1,\dotsc,x_n\in\bl$. The smallest such $M$ is called the $p$-convexity constant $M^{(p)}(\bl)$ of $\bl$. All $\Ell_p$-spaces are $p$-convex with $p$-convexity constant 1. All Banach lattices are 1-convex with 1-convexity constant 1.

A Banach lattice $\bl$ is said to be \emph{$\infty$-convex with $\infty$-convexity constant at most $M$} when
\[
\lrnorm{\bigvee_{i=1}^n\abs{x_i}}\leq M \bigvee_{i=1}^n\norm{x_i}
\]
for $x_1,\dotsc,x_n\in\bl$. The smallest such $M$ is the $\infty$-convexity constant $M^{(\infty)}(\bl)$ of $\bl$.

Some parts of the proofs of the following characterisations of $p$-convexity for $1\leq p\leq\infty$ are inspired by ideas in the proof of \cite[Theorem~6.1]{jardon-sanchez_laustsen_taylor_tradacete_troitsky:2022}.

We start with $p=1$.

\begin{proposition}\label{res:characterisation_1}
	Let $\bl$ be a Banach lattice. Then, for $x\in\bl$,
	\[
	\norm{x}=\max_{\mor}\norm{\mor(x)},
	\]
	where $\mor$ runs over all contractive vector lattice homomorphisms $\mor\colon\bl\to\Ell_1(\mu_\mor)$ for probability measures $\mu_\mor$.
\end{proposition}
\begin{proof}
	Take a non-zero $x_0\in E$. Choose $x_0^*\in\ball_{\bldual_+}$ such that $x_0^*(\abs{x_0})=\norm{x_0}$. We let $N_{x_0^*}\coloneqq\{x\in\bl:x_0^*(\abs{x})=0\}$ be its null ideal, and define $\bltwo$ to be the completion of $\bl/N(x_0^*)$ in the norm $\norm{x+N_{x_0^*}}\coloneqq x_0^*(\abs{x})$. The canonical vector lattice homomorphism from $E$ into $F$ is contractive. As any two disjoint positive elements in a quotient have disjoint positive lifts, $\bltwo$ is an AL-space, so that it can be identified with a concrete $\Ell_1$-space $L_1(\nu)$ on a set $\pset$. If we let $\mortwo\colon\bl\to\Ell_1(\nu)$ be the resulting contractive vector lattice homomorphism, then $\norm{x_0}=\norm{\mortwo(x_0 )}_{\Ell_1(\nu)}$. To obtain a probability measure, we set $S\coloneqq\{\pt\in\pset:[\mortwo(x_0)](\pt)\neq 0\}$, we let $\nu\!\restriction_S$ be the restriction of $\nu$ to $S$, and define a probability measure $\mu$ on $S$ by setting
	\[
	\mathrm{d}\mu\coloneqq\frac{\abs{\psi(x_0)}\!\restriction_S}{\norm{x_0}}\,\,\mathrm{d}\!\nu\!\restriction_S.
	\]
	We define  $\chi\colon\Ell_1(\nu)\to\Ell_1(\mu)$ by setting
	\[
	\chi(f)\coloneqq \frac{\norm{x_0}}{\abs{\psi(x_0) }\!\restriction_s} \,\,f\!\!\restriction_S
	\] for $f\in\Ell_1(\nu)$. Then $\chi$ is a contractive vector lattice homomorphism. Furthermore,  $\norm{(\chi\circ\psi)(x_0)}_{\Ell_1(\mu)}=\norm{\psi(x_0)}_{\Ell_1(\nu)}=\norm{x_0}$. The result now follows easily.
\end{proof}

For $1<p<\infty$, we have the following.

\begin{theorem}\label{res:characterisation_p}
	Let $1<p<\infty$ and let $M\geq 1$. The following are equivalent for a Banach lattice $\bl$:
	\begin{enumerate}
		\item\label{part:characterisation_p_1} $E$ is $p$-convex and $M^{(p)}(E)\leq M$;
		\item\label{part:characterisation_p_2} for $x\in\bl$,
		\[
		\norm{x}\leq M \sup_{\mor}\norm{\mor(x)},
		\]
		where $\mor$ runs over all contractive vector lattice homomorphisms $\mor\colon\bl\to\Ell_p(\mu_\mor)$ for probability measures $\mu_\mor$.
	\end{enumerate}
\end{theorem}

\begin{proof}
	We start with the easy proof that part~\ref{part:characterisation_p_2} implies part~\ref{part:characterisation_p_1}.\! Take $x_1,\dotsc,x_n\in\bl$. Using  that the codomains are $\Ell_p$-spaces in the third step, we then have
	\begin{align*}
		\lrnorm{\left(\sum_{i=1}^n\abs{x_i}^p\right)^{1/p}}&\leq M \sup_\mor\lrnorm{\mor\left(\left(\sum_{i=1}^n\abs{x_i}^p\right)^{1/p}\right)} \\
		&=M \sup_\mor \lrnorm{\left(\sum_{i=1}^n\abs{\mor(x_i)}^p\right)^{1/p}}\\
		&=M \sup_\mor \left(\sum_{i=1}^n\norm{\mor(x_i)}^p\right)^{1/p}\\	
		&\leq M \left(\sum_{i=1}^n\norm{x_i}^p\right)^{1/p}.
	\end{align*}
	
	We now prove that part~\ref{part:characterisation_p_1} implies part~\ref{part:characterisation_p_2}. Take $x\in\bl$. \cref{res:characterisation_1} furnishes a probability measure $\mu$ and a contractive vector lattice homomorphism $\mor\colon\bl\to\Ell_1(\mu)$ such that $\norm{x}_E=\norm{\mor(x)}_{\Ell_1(\mu)}$. Take $x_1,\dotsc,x_n\in\bl$. Then
	\begin{align*}
		 \lrnorm{\left(\sum_{i=1}^n\abs{\mor(x_i)}^p\right)^{1/p}}_{\Ell_1(\mu)}&=\lrnorm{\mor\left(\left(\sum_{i=1}^n\abs{x_i}^p\right)^{1/p}\right)}_{\Ell_1(\mu)}\\
		&\leq\lrnorm{\left(\sum_{i=1}^n\abs{x_i}^p\right)^{1/p}}_\bl\\
		&	\leq M \left(\sum_{i=1}^n\norm{x_i}_\bl^p\right)^{1/p}.
	\end{align*}
	 The Maurey-Nikishin factorisation theorem (see \cite[Theorem~7.1.2]{albiac_kalton_TOPICS_IN_BANACH_SPACE_THEORY:2006}) now yields a probability measure $\nu$, a vector lattice homomorphism $S\colon\bl\to\Ell_p(\nu)$ with $\norm{S}\leq M$, and an isometric embedding $j\colon \Ell_1(\nu)\to\Ell_1(\mu)$ such that the diagram
	
	\begin{equation*}
		\begin{tikzcd}
			\bl\arrow[r, "\mor"]\arrow[d, "S", swap]&  \Ell_1(\mu)
			\\ \Ell_p(\nu)\arrow[r,hookrightarrow,"i"]& \Ell_1(\nu)\arrow[u, "j", swap, hookrightarrow]
		\end{tikzcd}
	\end{equation*}
	is commutative. Here $i$ is the inclusion map, which is contractive as $\nu$ is a probability measure. We then have
	\begin{align*}
		\norm{x}_E&=\norm{\mor(x)}_{\Ell_1(\mu)}=\norm{(j\circ i\circ S)(x)}_{\Ell_1(\mu)}=\norm{(i\circ S)(x)}_{\Ell_1(\nu)}\leq \norm{ S(x)}_{\Ell_p(\nu)}.
	\end{align*}	
	On writing $S=M\cdot(S/M)$ we see that the inequality in part~\ref{part:characterisation_p_2} holds for $x$.
\end{proof}

To conclude, we cover the case $p=\infty$.

\begin{proposition}\label{res:characterisation_infinite_p}
	The following are equivalent for a Banach lattice $\bl$:
	\begin{enumerate}
		\item\label{part:characterisation_infinite_p_1} $E$ is $\infty$-convex with $\infty$-convexity constant 1;
		\item\label{part:characterisation_infinite_p_2} for $x\in\bl$,
		\[
		\norm{x}=\max_{\mor}\norm{\mor(x)},
		\]
		where $\mor$ runs over all contractive vector lattice homomorphisms $\mor\colon\bl\to\Ell_\infty(\mu_\mor)$ for probability measures $\mu_\mor$;
		\item\label{part:characterisation_infinite_p_3} for $x\in\bl$,
		\[
		\norm{x}=\max_{\mor}\norm{\mor(x)},
		\]
		where $\mor$ runs over all contractive vector lattice homomorphisms $\mor\colon\bl\to\RR$.
	\end{enumerate}
\end{proposition}

\begin{proof}
	It follows from the Kakutani representation theorem (see \cite[Theorem~1.b.6]{lindenstrauss_tzafriri_CLASSICAL_BANACH_SPACES_VOLUME_II_FUNCTION_SPACES:1979}, for example) that the $\infty$-convex Banach lattices with $\infty$-convexity constant 1 are, up to an isometric vector lattice homomorphism, precisely the Banach sublattices of $\cont(K)$ for some compact Hausdorff space $K$. On considering the point evaluations in such a functional representation, we see that part~\ref{part:characterisation_infinite_p_1} implies  part~\ref{part:characterisation_infinite_p_3}.
	
	Since $\RR$ is the $\Ell_\infty$-space for a probability measure on a one-point set, it is clear that part~\ref{part:characterisation_infinite_p_3} implies  part~\ref{part:characterisation_infinite_p_2}.
	
	The argument that part~\ref{part:characterisation_infinite_p_2} implies  part~\ref{part:characterisation_infinite_p_1} is similar to that in the first part of the proof of \cref{res:characterisation_p}.
\end{proof}

The three preceding results combine to yield the following.

\begin{corollary}\label{res:characterisation_all_p}
	Let $1\leq p\leq \infty$. The following are equivalent for a Banach lattice $\bl$:
	\begin{enumerate}
		\item\label{part:characterisation_all_p_1} $E$ is $p$-convex with $p$-convexity constant 1;
		\item\label{part:characterisation_all_p_2} for $x\in\bl$,
		\begin{equation}\label{eq:norming_equation}
		\norm{x}=\sup_{\mor}\norm{\mor(x)},
		\end{equation}
		where $\mor$ runs over all contractive vector lattice homomorphisms $\mor\colon\bl\to\Ell_p(\mu_\mor)$  for probability measures $\mu_\mor$.
	\end{enumerate}
When this is the case for $p=1$ or $p=\infty$, the supremum in \eqref{eq:norming_equation} is a maximum.
\end{corollary}
\begin{remark}
	For $1<p<\infty$, we shall see in \cref{rem:renorming} how \cref{res:characterisation_all_p} provides an alternate proof of the fact that a $p$-convex Banach lattice has an equivalent vector lattice norm in which it is $p$-convex with $p$-convexity constant~1.
\end{remark}

\subsection{Convexity types}

The definition of a convexity type uses the positively homogeneous functional calculus.
For $n\in\NN$, we let $\homon$ denote the set of all continuous, positively homogeneous, real-valued functions on $\RR^n$. We let $\funincrn$ denote the set of functions $f\colon \RR^n_+\to\mathbb R$ such that $f(0,\dotsc,0)=0$ and $f(t_1,\dotsc,t_n)\leq f(s_1,\dotsc,s_n)$ when $t_i\leq s_i$ for $i=1,\dots,n$. There are no regularity requirements for these increasing functions on the positive cone $\RR^n_+$ of $\RR^n$; it will become clear in the proof of \cref{res:quotients_C-convex} and \cref{res:convexity} why they are supposed to be increasing. The condition that $f(0,\dotsc,0)=0$, which implies that $f$ is positive, is not essential, but is necessary when one wants to ensure that the class of $\ct$-convex Banach lattices as defined below is non-void.

A \emph{convexity implication} $\ci$ is a triple $(\cia,\cic;n)$, where $n\in\NN$, $\cia\subseteq\homon$, and $\cic\subseteq\homon\times\funincrn$. We allow that $\cia=\emptyset$ but require that $\cic\neq\emptyset$. $\cia$ is the \emph{antecedent} of $\ci$; $\cic$ is its \emph{consequent}. A Banach lattice $\bl$ is said to \emph{satisfy $\ci$} if
\begin{equation}\label{eq:convexity_inequality}
	\norm{c_1(x_1,\dotsc,x_n)}\leq c_2(\norm{x_1},\dots,\norm{x_n})
\end{equation}
for every $(c_1,c_2)\in\cic$, whenever $x_1,\dotsc,x_n\in\bl$ are such that $a(x_1,\dotsc,x_n)\leq 0$ for every $a\in\cia$. When $\cia=\emptyset$, we require that \eqref{eq:convexity_inequality} hold for all $x_1,\cdots,x_n\in\bl$.\footnote{Thus $\bl$ satisfies $\ci=(\cia,\cic;n)$ if an implication is valid for $n$-tuples of its elements, the (possibly vacuously fulfilled) antecedent of which is expressed in terms of the elements $a$ of $\cia$, and the consequent of which is expressed in terms of the elements $(c_1,c_2)$ of $\cic$. This motivates the terminology and notation.}

A \emph{convexity type} $\ct$ is a (possibly empty) set of convexity implications. A Banach lattice is said to be \emph{$\ct$-convex} if it satisfies all convexity implications in $\ct$. The zero lattice is always $\ct$-convex and every Banach lattice is $\emptyset$-convex.

Before giving examples of convexity types, we elaborate a little on how the antecedent $\cia$ of a convexity implication $(\cia,\cic;n)$ can be used to express various conditions on elements $x_1,\cdots,x_n$ of a Banach lattice under which the resulting inequalities  from $\cic$ in \eqref{eq:convexity_inequality} are required to hold.

First of all, we note that, for $a\in\homon$, including both $a$ and $-a$ in $\cia$ also gives the condition $a(x_1,\dots,x_n)=0$ a place in the formalism.

The following example shows how the common conditions of being pairwise disjoint, positive, or both, all result from an suitable $\cia$.

\begin{example}\label{ex:antecedents_1}
	Take $n\in\NN$ and define $a_n\in\homon$ by setting $a_n(t_1,\dotsc,t_n)\coloneqq \sum_{i,j=1}^n\abs{t_i}\wedge\abs{t_j}$. Then $a_n(x_1,\dotsc,x_n)\leq 0$ if and only if the $x_i$ are pairwise disjoint.  If we set
	$a_n^\prime(t_1,\dotsc,t_n)\coloneqq \sum_{i=1}^n {\abs{t_i-\abs{t_i}}}$, then $a_n^\prime(x_1,\dotsc,x_n)\leq 0$ if and only if $x_1,\dotsc,x_n\in\blpos$.  Thus, for $\ci=(\{a_n\},\cic;n)$ (resp.\ $\ci=(\{a_n^\prime\},\cic;n))$, a Banach lattice $\bl$ satisfies $\ci$ when the resulting inequalities  from $\cic$ in \eqref{eq:convexity_inequality} hold for all pairwise disjoint (resp.\ positive) $x_1,\dotsc,x_n$ in $E$. For  $\ci=(\{a_n+a_n^\prime\},\cic;n)$, the inequalities are required to hold for all pairwise disjoint positive $x_1,\dotsc,n$.
	\end{example}

 Whereas the antecedents $\{a_n\}$, $\{a_n^\prime\}$, and $\{a_n^\prime+a_n^{\prime\prime}\}$ in \cref{ex:antecedents_1} are evidently relevant in practice, the formalism also captures other conditions that are quite natural. For example, if $n\geq 2$ and we set
	\[
	a_n^{\prime\prime}(t_1,\dotsc,t_n)\coloneqq \sum_{i=1}^n \abs{t_i-\abs{t_i}} +\sum_{i=1}^n \abs{t_i}\wedge\abs{t_j} + \lrabs{t_n-\sum_{i=1}^{n-1}t_i},
	\]
	then $a^{\prime\prime}(x_1,\dotsc,x_n)\leq 0$ if and only if the $x_i$ are positive, pairwise disjoint, and $x_n=x_1+\dotsb x_{n-1}$. If $n\geq 2$ and we set
	\[
	a_n^{\prime\prime\prime}(t_1,\dotsc,t_n)\coloneqq \sum_{i=1}^n {\abs{t_i-\abs{t_i}}}+\sum_{i=1}^n \abs{t_i-{t_{i-1}}-\abs{t_i-t_{i-1}}},
	\]
	then $a^{\prime\prime\prime}(x_1,\dotsc,x_n)\leq 0$ if and only if $0\leq x_1\leq x_2\leq\dotsb\leq x_n$.
	
	Finally, we remark that it is also possible to obtain more `exotic' conditions on $x_1,\dotsc,x_n$ by letting $\cia$ consist of more than one function, and also including functions that are not in the vector lattice generated by the coordinate functionals. For example, take $1\leq p,q<\infty$ and $M>0$, and set
	\[
	a_n^{\prime\prime\prime\prime}(t_1,\dotsc,t_n)\coloneqq \left(\sum_{i=1}^n \abs{t_i}^p\right)^{1/p}-M\left(\sum_{i=1}^n \abs{t_i}^q\right)^{1/q}.
	\]
	If $\ci=(\{	a_n^{\prime\prime\prime},a_n^{\prime\prime\prime\prime}\},\cic;n)$ is a convexity implication, then a Banach lattice $\bl$ satisfies $\ci$ when the resulting inequalities from $\cic$ in \eqref{eq:convexity_inequality} hold for all $0\leq x_1\leq\dotsb\leq x_n$ in $\bl$ such that $ \left(\sum_{i=1}^n \abs{x_i}^p\right)^{1/p}\leq M\left(\sum_{i=1}^n \abs{x_i}^q\right)^{1/q}$.

We now continue with some examples of convexity types.

\begin{example}\label{ex:convexity_type_1}
	Let $M\geq 1$. Take $1\leq p<\infty$. If we set $c_n^p(t_1,\ldots,t_n)\coloneqq\left(\sum_{i=1}^n\abs{t_i}^p\right)^{1/p}$ for $n\geq 1$, and take
	\[
	\ct^{p,M}\coloneqq \big\{(\emptyset,\{(c_n^p, M c_n^p)\};n):n=1,2,\ldots \big\},
	\]
	then the $\ct^{p,M}$-convex Banach lattices are the $p$-convex Banach lattices with $p$-convexity constant at most $M$. If we take		
	\[
	\ct^{p,M,\mathrm{u}}\coloneqq \big\{(\{a_n\},\{(c_n^p, M c_n^p)\};n):n=1,2,\ldots \big\},
	\]
	where the $a_n$ are as in \cref{ex:antecedents_1}, then the $\ct^{p,M,\mathrm{u}}$-convex Banach lattices are the Banach lattices satisfying a so-called upper $p$-estimate with constant $M$.
\end{example}

\begin{example}\label{ex:convexity_type_2}
	Let $M\geq 1$. If we set $c_n^\infty (t_1,\ldots,t_n)\coloneqq\sum_{i=1}^n\abs{t_i}$ for $n\geq 1$, and take
	\[
	\ct^{\infty,M} \coloneqq \big\{(\emptyset,\{(c_n^\infty, M c_n^\infty)\};n):n=1,2,\ldots \big\},
	\]
	then the $\ct^{\infty,M}$-convex Banach lattices are the $\infty$-convex Banach lattices with $\infty$-convexity constant at most $M$.
\end{example}

\begin{example}\label{ex:convexity_type_3}
	Let $M\geq 1$. Take $1\leq p,q<\infty$. With $c_n^p$ and $c_n^q$ as in \cref{ex:convexity_type_1}, if we take
	\[
	\ct^{p,q,M}\coloneqq \big\{(\emptyset,\{(c_n^q, M c_n^p)\};n):n=1,2,\ldots \big\},
	\]
	then the $\ct^{p,q,M}$-convex Banach lattices are the so-called $(p,q)$-convex Banach lattices with $(p,q)$-convexity constant at most $M$.\footnote{We refer to the literature for the non-trivial facts that the zero lattice is the only $(p,q)$-convex Banach lattice when $1\leq q<p$, and that the
	the $(p,q)$-convex Banach lattices coincide with the  $(p,q^\prime)$-convex ones when $p<q,q^\prime<\infty$.} If we let
	\[
	\ct^{p,q,M,\mathrm{u}}\coloneqq \big\{(\{a_n\},\{(c_n^q, M c_n^p)\};n):n=1,2,\ldots \big\},
	\]
	where the $a_n$ are as in \cref{ex:antecedents_1}, then the $\ct^{p,q,M,\mathrm{u}}$-convex Banach lattices are those that could be said to  satisfy an upper $(p,q)$-estimate with constant $M$.
\end{example}

Clearly, a Banach sublattice of a $\ct$-convex Banach lattice is also $\ct$-convex. For a quotient, this is not automatic. A sufficient condition is that elements in the quotient that satisfy all the conditions from an antecedent should have lifts as in the following result. It has its roots in \cite[Lemma~4.7]{jardon-sanchez_laustsen_taylor_tradacete_troitsky:2022}, where such lifts always exist.

\begin{lemma}\label{res:quotients_C-convex}
Let $\ct$ be a convexity type, let $\bl$ be a $\ct$-convex Banach lattice, let $I$ be a closed order ideal in $\bl$, and let $q\colon\bl\to\bl/I$ be the quotient vector lattice homomorphism. Suppose that the following hold for every $\ci=(\cia,\cic;n)$ in $\ct$ :
\begin{enumerate}
	\item $c_2$ is continuous for all $(c_1,c_2)\in\cic$;
	\item for all $\widetilde x_1,\dotsc,\widetilde x_n\in \bl/I$ such that $a(\widetilde x_1,\dotsc,\widetilde x_n)\leq 0$ for all $a\in\cia$, and for all $\varepsilon>0$, there exist $x_1,\dotsc,x_n\in\bl$ such that
	\begin{enumerate}
		\item $q(x_k)=\widetilde x_k$;
		\item $a(x_1,\dotsc,x_n)\leq 0$ for all $a\in\cia$;
		\item $\norm{x_k}\leq\norm{\widetilde x_k}+\varepsilon$ for $k=1\dotsc,n$.
	\end{enumerate}
\end{enumerate}
Then $\bl/I$ is $\ct$-convex.
\end{lemma}

\begin{proof}
Take $\ci=(\cia,\cic;n)$ in $\ct$, and suppose that $\widetilde x_1,\dotsc,\widetilde x_n\in \bl/I$ are such that $a(\widetilde x_1,\dotsc,\widetilde x_n)\leq 0$ for all $a\in\cia$. Take $(c_1,c_2)\in\cic$, and let $\varepsilon>0$. Choose  $x_1,\dotsc,x_n\in\bl$ as in the statement. Using that $\bl$ is $\ct$-convex and that $c_2$ is increasing, we have
\begin{align*}
\norm{c_1(\widetilde x_1,\dotsc,\widetilde x_n)}&=\norm{q(c_1(x_1,\dotsc,x_n))}
\leq \norm{(c_1(x_1,\dotsc,x_n))}\\
&\leq c_2(\norm{x_1},\dotsc,\norm{x_n})
\leq c_2(\norm{\widetilde x_1}+\varepsilon,\dotsc,\norm{\widetilde x_n}+\varepsilon).
\end{align*}
As $c_2$ is supposed to be continuous and $\varepsilon$ is arbitrary, we conclude that $\bl/I$ is $\ct$-convex.
\end{proof}

We now come to the key fact about convexity types, to be used in \cref{sec:free_banach_lattices_over_pre-ordered_banach_spaces} for an easy proof that the candidate free Banach lattice as constructed there does indeed have the desired convexity type.

\begin{lemma}\label{res:convexity}
	Let $\bl$ be a Banach lattice and let $\ct$ be a convexity type. Suppose that, for all $x\in\bl$,
	\begin{equation*}\label{eq:convexity}
		\norm{x}=\sup_\mor\norm{\mor(x)},
	\end{equation*}
	where $\mor$ runs over a class of contractive vector lattice homomorphisms $\mor\colon\bl\to\bl_\mor$ into Banach lattices $\bl_\mor$. If all $\bl_\mor$ are $\ct$-convex, then $\bl$ is $\ct$-convex.
\end{lemma}

\begin{proof}
	If $\ct=\emptyset$, there is nothing to prove, so we suppose that $\ct$ is non-empty. Take a convexity implication $\ci=(\cia,\cic;n)$ in $\ct$. Suppose that $x_1,\cdots,x_n\in\bl$ are such that  $a(x_1,\ldots,x_n)\leq 0$ for all $a\in\cia$; this is vacuously fulfilled if $\cia=\emptyset$. Take a $\mor\colon\bl\to\bl_\mor$ as in the statement. Then $a(\mor(x_1),\ldots,\mor(x_n))=\mor(a(x_1,\ldots,x_n))\leq 0$ for all $a\in\cia$. Since $\bl_\mor$ satisfies $\ci$, we have 	
	\[
	\norm{c_1(\mor(x_1),\dotsc,\mor(x_n))}\leq c_2(\norm{\mor(x_1)},\dots,\norm{\mor(x_n)})
	\]
	for all $(c_1,c_2)\in\cic$. Using the fact that $c_2$ is increasing in the final step, we then see that, for all $(c_1,c_2)\in\cic$,
	\begin{align*}
		\norm{c_1(x_1,\dotsc,x_n)}&=\sup_\mor \norm{\mor\left(c_1(x_1,\dotsc,x_n)\right)}\\
		&=\sup_\mor \norm{c_1(\mor(x_1),\dotsc,\mor(x_n))}\\
		&\leq\sup_\mor c_2(\norm{\mor(x_1)},\dots,\norm{\mor(x_n)})\\
		&\leq c_2(\norm{x_1},\dotsc,\norm{x_2}).
	\end{align*}
	Hence $\bl$ satisfies $\ci$, as required.
\end{proof}

\begin{remark}\label{rem:renorming}
	Let $1< p<\infty$, and let $\bl$ be a $p$-convex Banach lattice. \cref{res:characterisation_p} then shows that $\norm{x}\leq M^{(p)}(E)\sup_\mor\norm{\mor(x)}$, where $\mor$ runs over all contractive vector lattice homomorphisms $\mor\colon\bl\to\Ell_p(\mu_\mor)$ for probability measures $\mu_\mor$. For $x\in\bl$, set $\norm{x}^\prime\coloneqq \sup_\mor\norm{\mor(x)}$. Since  $\norm{x}^\prime\leq\norm{x}\leq M^{(p)}(E)\norm{x}^\prime$,  $\norm{\,\cdot\,}^\prime$ is a complete vector lattice norm on $\bs$ that is equivalent to $\norm{\,\cdot\,}$. As all codomains of the $\varphi$ are $p$-convex with $p$-convexity constant 1, \cref{res:convexity} shows that the same is true for $\bl$ in the norm $\norm{\,\cdot\,}^\prime$. We have thus retrieved a part of \cite[Proposition~1.d.8]{lindenstrauss_tzafriri_CLASSICAL_BANACH_SPACES_VOLUME_II_FUNCTION_SPACES:1979}.
\end{remark}

\begin{remark}
With \cref{ex:antecedents_1} in mind, it is not difficult to see how every convexity condition $\ccond$ as defined in \cite[Section~3]{jardon-sanchez_laustsen_taylor_tradacete_troitsky:2022} yields a convexity type $\ct$, such that the $\ccond$-convex Banach lattices as defined in \cite{jardon-sanchez_laustsen_taylor_tradacete_troitsky:2022} are precisely the $\ct$-convex Banach lattices as defined in the current paper. The present set-up is, however, more general, both where the antecedent and where the consequent are concerned.  Whereas the conditions on $x_1,\dots,x_n$ that are used in \cite{jardon-sanchez_laustsen_taylor_tradacete_troitsky:2022} are those of being pairwise disjoint or being positive, corresponding to the antecedents $\{a_n\}$ resp.\ $\{a_n^\prime\}$ in \cref{ex:antecedents_1}, the present formalism allows a variety of possible antecedents, which can be sets of arbitrary cardinality. On the consequent side, \cref{ex:convexity_type_3} shows that $(p,q)$-convexity is covered by the present formalism, whereas it is not by that in \cite{jardon-sanchez_laustsen_taylor_tradacete_troitsky:2022}.
\end{remark}

It can happen that the zero lattice is the only $\ct$-convex Banach lattice.  As an example where this is the case, take $n=1$, define $c_1(t_1)=2t_1$ and $c_2(t_1)=t_1$, and set $\ct=\big\{(\emptyset,\{(c_1,c_2)\};1)\big\}$. There is an easy criterion for $\ct$ to avoid such a collapse of \cref{sec:free_banach_lattices_over_pre-ordered_banach_spaces} to the study of the zero lattice. On considering one-dimensional Banach sublattices, we see that the existence of non-zero $\ct$-convex Banach lattices is equivalent to $\RR$ being $\ct$-convex. Therefore, {\bf we assume in the sequel that $\boldsymbol{\RR}$ is $\boldsymbol{\ct}$-convex.}

\section{Free vector lattices over pre-ordered vector spaces}\label{sec:free_vector_lattices_over_pre-ordered_vector_spaces}

\noindent The existence of free Banach lattices over pre-ordered Banach spaces in \cref{sec:free_banach_lattices_over_pre-ordered_banach_spaces} is a rather straightforward consequence of the existence of free vector lattices over pre-ordered vector spaces. The current sections is concerned with these. We show that they exist and{\textemdash}although this is not needed for the existence of the free Banach lattices{\textemdash}develop their theory a little further. For clarity, we recall that in the present paper vector lattices are not supposed to be Archimedean.

We apply \cref{def:free_object} in the following context. For $\catone$, we take the pre-ordered vector spaces as objects and the positive linear maps as morphisms. For  $\cattwo$, we take the vector lattices as objects and the vector lattice homomorphisms as morphisms. We define the functor $\functor$ from $\cattwo$ to $\catone$ by sending a vector lattice $\vl$ to the pre-ordered vector space $(\vl,\pospartofset{\vl})$ and viewing a vector lattice homomorphism as a positive linear map. Simplifying the notation and terminology a little, this leads to the following definition.

\begin{definition}\label{def:free_vector_lattice_over_a_pre-ordered_vector_space}
	A \emph{free vector lattice over the pre-ordered vector space $\povs$} is a pair $(j,\FVLpovs)$, where
	$\FVLpovs$ is a vector lattice and $j\colon\vs\to\FVLpovs$ is a positive linear map with the property that, for every positive linear map $\mor\colon\vs\to\vltwo$ into a vector lattice $\vltwo$, there exists a unique vector lattice homomorphism $\facmor\colon\FVLpovs\to\vltwo$ such that the diagram
	\begin{equation}\label{dia:free_vector_lattice}
		\begin{tikzcd}
			\vs\arrow[r, "j"]\arrow[dr, "\mor", swap]& \FVLpovs \arrow[d, "\facmor"]
			\\ & \vltwo
		\end{tikzcd}
	\end{equation}
	commutes.
\end{definition}

\begin{remark}\label{rem:free_vector_lattice_over_a_pre-ordered_vector_space}\quad
	\begin{enumerate}
		\item\label{part:free_vector_lattice_over_a_pre-ordered_vector_space_1} If $\vspos=\zerospace $, then every linear map $\mor\colon\vs\to\vltwo$ is positive. Hence $(j,\FVL[(\vs,\zerospace )])$ is what is generally called the free vector lattice over the vector space $\vs$.
		\item If $\FVLpovs$ exists, then standard argument shows that it is generated by $j(\vs)$ as a vector lattice.
		\item\label{part:free_vector_lattice_over_a_pre-ordered_vector_space_2} If $\FVLpovs$ exists, then it is clear that the map $j$ in diagram  \eqref{dia:free_vector_lattice} is injective if and only if the positive linear maps $\mor\colon \vs\to\vltwo$ into vector lattices $\vltwo$ separate the points of $\vs$. When $\vspos=\zerospace $, the linear maps into $\RR$ already do this, so that the free vector lattice over a vector space $\vs$ will contain $\vs$ as a linear subspace. This is not always the case. When $\vspos=\vs$, the only positive linear map $\mor$ in diagram \eqref{dia:free_vector_lattice}  is the zero map. Hence $\FVL[(V,V)]=\zerospace $ and $j$ is the zero map. We shall have more to say about the kernel of $j$ in \cref{res:free_vector_lattice_over_pre-ordered_vector_space_as_lattice_of_functions,res:free_vector_lattice_over_pre-ordered_banach_space_as_lattice_of_functions}.
	\end{enumerate}
\end{remark}

\begin{theorem}\label{res:existence_of_free_vector_lattice}
	Let $\povs$ be a pre-ordered vector space. There exists a free vector lattice $(j,\FVLpovs)$ over $\povs$.
\end{theorem}

\begin{proof}
	There exists a free vector lattice $(j,\FVL[(\vs,\zerospace )]))$ over $(\vs,\zerospace )$. This is a part of  \cite[Theorem~6.2]{de_jeu:2021}. Alternatively, one can infer this from the well-known fact that the free vector lattice over a set exists (see \cites{baker:1968,birkhoff:1942,bleier:1973,weinberg:1963}); the free vector lattice over a basis of $\vs$ is then as needed. \footnote{In both approaches, the existence ultimately relies on the existence theorem in universal algebra for free objects in equational classes over sets; see \cite[Corollary to  Theorem~2.10]{jacobson_BASIC_ALGEBRA_II_SECOND_EDITION:1989}, for example. We refer to \cite{de_jeu:2021} for a detailed and self-contained exposition of this material in the context of vector lattices and vector lattice algebras.}.
	
	Take the order ideal $I$ of $\FVL[(\vs,\zerospace )]$ that is generated by $\{j(v)^-:v\in \vspos\}$. Let $q\colon\FVL[(\vs,\zerospace )]\to\FVL[(\vs,\zerospace )]/I$ denote the quotient vector lattice homomorphism. We claim that $(q\circ j,\FVL[(\vs,\zerospace )]/I)$ is a free vector lattice over the pre-ordered vector space $\povs$.
	
	Take $v\in\vspos$. Then $[(q\circ j)(v)]^{\,-}=[q(j(v))]^-=q(j(v)^{\,-})=0$. Hence $q\circ j$ is positive, as required.
	
	Let $\mor\colon\vs\to\vltwo$ be a positive linear map into a vector lattice $\vltwo$. There exists a unique vector lattice homomorphism $\facmor\colon\FVL[(\vs,\zerospace )]\to\vltwo$ such that $\facmor\circ j=\mor$. Take $v\in\vspos$. Using the positivity of $\mor$ in the final step, we have $\facmor\big(j(v)^-\big)=\left[\facmor(j(v))\right]^{\,-}=\mor(v)^{\,-}=0$.
	Hence $I\subseteq\ker\mor$, which implies that there exists a unique vector lattice homomorphism $\fac{\facmor}\colon\FVL[(\vs,\zerospace )]/I\to\vltwo$ such that $\fac{\facmor}\circ q =\facmor$. Since the diagram
	\begin{equation*}
		\begin{tikzcd}
			\vs\arrow[r, "j"]\arrow[dr,"\mor"]& \FVL[(\vs,\zerospace )] \arrow[r,"q"]\arrow[d, "\facmor"]&\FVL[(\vs,\zerospace )]/I\arrow[dl,"\fac{\facmor}"]\\
			&\vltwo&
		\end{tikzcd}
	\end{equation*}
	is commutative, we have $\fac{\facmor}\circ(q\circ j)=\mor$. It remains to show uniqueness. Suppose that $\psi\colon\FVL[(\vs,\zerospace )]/I\to\vltwo$ is a vector lattice homomorphism such that  $\psi \circ(q\circ j)=\mor$. Then $(\psi\circ q)\circ j=\mor$, so $ \psi\circ q=\facmor$ and then $\psi=\fac{\facmor}$.
\end{proof}

\begin{remark}\label{rem:projection}
	Suppose that $\vs$ itself is a vector lattice. Then it is easy to see that  $(j\circ\fac{\id})\colon\FVLpovs\to\FVLpovs$, where $\id\colon\vs\to\vs$ is the identity map, is a positive projection onto $j(\vs)$.
\end{remark}

\cref{res:existence_of_free_vector_lattice} is sufficient as the starting point for the  construction of free Banach lattices over pre-ordered Banach spaces in \cref{sec:free_banach_lattices_over_pre-ordered_banach_spaces}. Nevertheless, we develop the material a little further, for its own interest and also to put later results in this paper and existing results and constructions in the literature into perspective.

It is well known that the free vector lattice over a set can be realised as a vector lattice of functions. The first occurrence of this result appears to be  \cite[Theorem~2.4]{baker:1968}; with a different proof it can also be found as \cite[Theorem~2.3]{bleier:1973}. The latter proof is based on convex analysis, which is also used in \cite[Theorem~2.3]{de_jeu_UNPUBLISHED:2020b} to show that the free vector lattice over a vector space can be realised as a vector lattice of functions on any linear subspace of its dual that separates its points.  In our next result, we generalise the latter in a more detailed form to the pre-ordered context.

\begin{theorem}\label{res:free_vector_lattice_over_pre-ordered_vector_space_as_lattice_of_functions}
	Let $\povs$ be a pre-ordered vector space.  Take a linear subspace $\lsdual$ of $\vsdual$ that separates the points of $V$, and let $\sigma(\vs,L^\sharp)$ denote the weak topology that $L^\sharp$ induces on $\vs$. Define the positive linear map $\Psi:V\to\Fun(L^\sharp_+)$ by setting
	\[
	[\Psi(v)](l^\sharp)\coloneqq l^\sharp(v)
	\]
	for $v\in V$ and $l^\sharp\in L^\sharp_+$. Then:
	 \begin{enumerate}
	 	\item for $v\in\vl$, $\Psi(v)\geq 0$ if and only if $v\in \overline{\vspos}^{\sigma(\vs,L^\sharp)}$;  \label{part:free_vector_lattice_over_pre-ordered_vector_space_as_lattice_of_functions_1}	
		\item $\Ker \Psi=\overline{\vspos}^{\sigma(\vs,L^\sharp)}\cap -\overline{\vspos}^{\sigma(\vs,L^\sharp)}$;\label{part:free_vector_lattice_over_pre-ordered_vector_space_as_lattice_of_functions_2}
		\item $L_+^\sharp=\zerospace$ if and only if $\overline{\vspos}^{\sigma(\vs,L^\sharp)}=\vs$. \label{part:free_vector_lattice_over_pre-ordered_vector_space_as_lattice_of_functions_3}		
	\end{enumerate}
	Take a free vector lattice $(j,\FVLpovsclosure)$ over $\povsclosure$. Then the unique vector lattice homomorphism $\fact{\Psi}$ making the diagram	
	\begin{equation*}\label{dia:free_vector_lattice_as-function_lattice}
		\begin{tikzcd}
			\vs\arrow[r, "j"]\arrow[dr, "\Psi", swap]& \FVLpovsclosure \arrow[d, "\fac{\Psi}"]
			\\ & \Fun(L^\sharp_+)
		\end{tikzcd}
	\end{equation*}
	commutative is injective, establishing a vector lattice isomorphism between the free vector lattice $\FVLpovsclosure$ and the Archimedean vector sublattice $\funlattice$ of $\Fun(L^\sharp_+)$ that is generated by $\Psi(V)$. Consequently, when we view $\Psi$ as a map from $\vs$ into $\funlattice$, $(\Psi,\funlattice)$ is a free vector lattice over $\povsclosure$.  The map $\Psi$ is injective if and only if $\overline{\vspos}^{\sigma(\vs,L^\sharp)}$ is a cone. For $l^\sharp\in\ls^\sharp_+$, the unique vector lattice homomorphism $\fac{l^\sharp}$ making the diagram
	\begin{equation*}\label{dia:free_vector_lattice_as-function_lattice_point_evaluation}
		\begin{tikzcd}
			\vs\arrow[r, "\Psi"]\arrow[dr, "l^\sharp", swap]& \funlattice \arrow[d, "\fac{l^\sharp}"]
			\\ & \mathbb R
		\end{tikzcd}
	\end{equation*}
commutative is given by $\fac{l^\sharp}(\freeelt)=\freeelt(l^\sharp)$ for $\freeelt\in\funlattice$, so that the $\fac{l^\sharp}$ for $l^\sharp\in\lsdualpos$ separate the points of $\funlattice$.
\end{theorem}

When $\vspos=\zerospace$, \cref{res:free_vector_lattice_over_pre-ordered_vector_space_as_lattice_of_functions} applies to every linear subspace $L^\sharp$ of $V^\sharp$ separating the points of $\vs$, yielding \cite[Theorem~2.3]{de_jeu_UNPUBLISHED:2020b}.

\begin{proof}
	We start with part~\ref{part:free_vector_lattice_over_pre-ordered_vector_space_as_lattice_of_functions_1},  from which the  parts~\ref{part:free_vector_lattice_over_pre-ordered_vector_space_as_lattice_of_functions_2} and~\ref{part:free_vector_lattice_over_pre-ordered_vector_space_as_lattice_of_functions_3} are immediate. It is clear that $\Psi(v)\geq 0$ when $v\in \overline{\vspos}^{\sigma(\vs,L^\sharp)}$. For the converse, take $v_0\in\vs$ such that $\Psi(v_0)\geq 0$. Suppose that $v_0\notin\overline{\vspos}^{\sigma(\vs,L^\sharp)}$. Then \cite[Theorems~3.4 and~3.10]{rudin_FUNCTIONAL_ANALYSIS_SECOND_EDITION:1991} show that there exist $\alpha,\beta\in\RR$ with $\alpha<\beta$ and $l_0^\sharp\in L^\sharp$, such that  $l_0^\sharp(v_0)<\alpha$ and $l_0^\sharp(v)>\beta$ for $v\in\overline{\vspos}^{\sigma(\vs,L^\sharp)}$. Since $\overline{\vspos}^{\sigma(\vs,L^\sharp)}$ is invariant under multiplication by non-negative scalars, we must have $l_0^\sharp(v)\geq 0$ for $v\in\overline{\vspos}^{\sigma(\vs,L^\sharp)}$, i.e., $l_0^\sharp\in L^\sharp_+$. We also have $l_0^\sharp(v_0)<\alpha<\beta<l_0^\sharp(0)=0$. This contradicts that $\Psi(v_0)\geq 0$. Thus $v_0\in\overline{\vspos}^{\sigma(\vs,L^\sharp)}$, as required.
	
	We now show that $\fact{\Psi}$ is injective, with a proof in the spirit of Bleier's paper \cite{bleier:1973}.  Its main part consist of establishing the following preparatory claim:
	
	\noindent\emph{Suppose that $v_1,\dotsc,v_q\in \vs$ are such that $\bigwedge_{r=1}^q \Psi(v_r)\leq 0$ in $\Fun(L^\sharp_+)$. Then $\bigwedge_{r=1}^q j(v_r)\leq 0$ in $\FVLpovsclosure$.}
	
	As the first step to see this, suppose that the intersection of the convex hull of $\{v_1,\dotsc,v_q\}$ and $-\overline{\vspos}^{\sigma(\vs,L^\sharp)}$ were empty. As this convex hull, being the continuous image of a finite dimensional simplex, is compact in the $\sigma(\vs,L^\sharp)$-topology, there exist $\alpha,\beta\in\RR$ with $\alpha<\beta$ and $l_0^\sharp\in L^\sharp$, such that  $l_0^\sharp(v_r)<\alpha$ for $r=1,\dotsc,q$ and $l_0^\sharp(-v)>\beta$ for $v\in\vspos$. As earlier, this implies that $-l_0^\sharp\in L^\sharp_+$ and that $\alpha<\beta<0$. We then have $\bigwedge_{r=1}^q (-l_0^\sharp)(v_r)>-\alpha>0$, contradicting that $\bigwedge_{r=1}^q \Psi(v_r)\leq 0$. We conclude that the intersection is not empty.
	
	As the second step, we can now choose non-negative $t_1,\dotsc,t_q$ with $\sum_{p=1}^q t_p=1$ such that $\sum_{p=1}^q t_p v_p\leq 0$. We have $t_p\bigwedge_{r=1}^q j(v_r)\leq t_p j(v_p)$ for $p=1,\dotsc,q$, so that
	\[
	\bigwedge_{r=1}^q j(v_r)=\sum_{p=1}^q \left(t_p \bigwedge_{r=1}^q j(v_r)\right)\leq \sum_{p=1}^q t_p j(v_p)=j\left(\sum_{p=1}^q t_pv_p\right)\leq 0,
	\]
	where the positivity of $j$ was used in the final step. Our claim has been established.
	
	We can now show the injectivity of $\fact{\Psi}$. Suppose that $\freeelt\in \FVLpovsclosure$ and that $\fact{\Psi}(\freeelt)=0$. Since $\FVLpovsclosure$ is generated as a vector lattice by its linear subspace $j(\vs)$, there exist an integer $N\geq 1$, integers $M_k\geq 1$ for $k=1,\dotsc,N$, and elements $v_{kl}$ of $\vs$ for $k=1,\dotsc,N$ and $l=1,\dotsc,M_k$ such that $\freeelt=\bigvee_{k=1}^N\bigwedge_{l=1}^{M_k} j(v_{kl})$; see \cite[Exercise~4.1.8]{aliprantis_burkinshaw_POSITIVE_OPERATORS_SPRINGER_REPRINT:2006}, for example. Hence $\fact{\Psi}(\freeelt)=\bigvee_{k=1}^N\bigwedge_{l=1}^{M_k} \fact{\Psi}(\left(j(v_{kl})\right)=\bigvee_{k=1}^N\bigwedge_{l=1}^{M_k} \Psi(v_{kl})=0$. Hence $\bigwedge_{l=1}^{M_k} \Psi(v_{kl})\leq 0$ for $k=1,\dotsc,N$. By the claim, we also have  $\bigwedge_{l=1}^{M_k} j(v_{kl})\leq 0$ for $k=1,\dotsc,N$, implying that $\freeelt \leq 0$. The fact that $\fact{\Psi}(-\freeelt)=0$ yields that also $-\freeelt\leq 0$. Hence $\freeelt=0$, as desired.
	
	The remaining statements are now easily verified.
	
\end{proof}

Clearly, the vector lattice $\funlattice$ in \cref{res:free_vector_lattice_over_pre-ordered_vector_space_as_lattice_of_functions} consists of positively homogeneous functions on $L\sp\sharp_+$. Let $\pobs$ be a pre-ordered Banach space.  \cref{res:free_vector_lattice_over_pre-ordered_vector_space_as_lattice_of_functions} can be applied with $\vs=\bs$ and $L^\sharp=\bsdual$.
The vector lattice of positively homogeneous function on $\bsdualpos$ is isomorphic to that of positively homogenous functions on $\ballbsdualpos$ via the restriction map. With this in mind, and recalling the definition of the $\delta_x$ in \eqref{eq:delta_definition}, the following is an immediate consequence of \cref{res:free_vector_lattice_over_pre-ordered_vector_space_as_lattice_of_functions}.

\begin{corollary}\label{res:free_vector_lattice_over_pre-ordered_banach_space_as_lattice_of_functions}
	Let $\pobs$ be a pre-ordered Banach space. Let $\funlattice$ denote the vector sublattice of $\contballposh$ that is generated by $\{\delta_x: x\in \bs\}$, and define $\delta\colon\bs\to\funlattice$ by setting $\delta(x)\coloneqq \delta_x$ for $x\in\bs$. Then:
	\begin{enumerate}
		\item for $x\in\bs$, $\delta_x\geq 0$ if and only if $x\in\overline{\bspos}$:
		\item $\Ker \delta=\overline{\bspos}\cap-\overline{\bspos}$;
		\item $\ballbsdualpos=\zerospace$ if and only if $\overline{\bspos}=\bs$;
		\item \begin{enumerate}
			\item $(\delta,\funlattice)$ is a free vector lattice over the pre-ordered vector space $\pobsclosure$;
			\item the map $\delta$ is injective if and only if $\overline{\bspos}$ is a cone;
			\item for $\varphi\in\ballbsdualpos$, the unique vector lattice homomorphism $\facmor$ making the diagram
			\begin{equation*}\label{dia:free_vector_lattice_as-function_lattice_point_evaluation_bs}
				\begin{tikzcd}
					\bs\arrow[r, "\delta"]\arrow[dr, "\mor", swap]& \funlattice \arrow[d, "\facmor"]
					\\ & \mathbb R
				\end{tikzcd}
			\end{equation*}
			commutative is given by $\facmor(\freeelt)=\freeelt(\mor)$ for $\freeelt\in\funlattice$, so that the $\facmor$ for $\mor\in\ballbsdualpos$ separate the points of $\funlattice$.
				\end{enumerate}
			\end{enumerate}
\end{corollary}

\begin{remark}\label{rem:escaped}
When $\bspos=\zerospace$, \cref{res:free_vector_lattice_over_pre-ordered_banach_space_as_lattice_of_functions} gives a realisation as a vector lattice of functions on $\ballbsdual$ of what is generally called the free vector lattice over the vector space $\bs$; this was already observed in \cite{de_jeu_UNPUBLISHED:2020b}. Using this fact, which is a special case of \cref{res:free_vector_lattice_over_pre-ordered_vector_space_as_lattice_of_functions}, could have simplified later papers on free Banach lattices somewhat. In \cite[p.18]{oikhberg_taylor_tradacete_troitsky:2024}, for example, the lattice-linear functional calculus is used to proved that $(\delta,\funlattice)$ has the pertinent universal property for $(\bs,\zerospace)$ with respect to Archimedean vector lattices as codomains $F$ in diagram \ref{dia:free_vector_lattice}. As is now clear, the Archimedean property can be omitted, and the lattice-linear functional calculus is not needed.
\end{remark}

Although they are not at the core of the present paper, the following consequences of \cref{res:free_vector_lattice_over_pre-ordered_vector_space_as_lattice_of_functions} should not be left unnoticed. The first one generalises \cite[Theorem~2.4]{baker:1968} and \cite[Theorem~2.3]{bleier:1973} on the realisation of the free vector lattice over a set as a function lattice.

\begin{corollary}\label{res:free_vector_lattice_over_arbitrary_pointred_set}
	Let $S$ be a non-empty set, and let $\pospartofset{S}$ be a \uppars{possibly empty} subset of $S$. Set
	\[
	F\coloneqq\{f\colon S\to\RR: f\text{ is positive on } \pospartofset{S}\}.
	\]
	For $s\in S$, define $j(s)\colon F\to\RR$ by setting
	\[
	[j(s)] (f)\coloneqq f(s)
	\]
	for $f\in F$. Let $\funlattice$ be the vector sublattice of $\Fun(F)$ that is generated by $j(S)$. Then $j\colon S\mapsto\funlattice$ is positive on $\pospartofset{S}$ and, for every map $\mor \colon S\to\vl$ into a vector lattice $\vl$ that is positive on $\pospartofset{S}$, there exists a unique vector lattice homomorphism making the diagram
	\begin{equation*}
		\begin{tikzcd}
			S \arrow[r, "j"]\arrow[dr, "\mor", swap]& \funlattice\arrow[d, "\facmor"]
			\\ & \vl
		\end{tikzcd}
	\end{equation*}
	commutative.
	\end{corollary}
	
	\begin{proof}
		Let $\vs$ be the free vector space over $S$. It has a basis $\{e_s:s\in S\}$, and we let $\pospartofset{\vs}$ be the cone generated by  $\{e_s:s\in \pospartofset{S}\}$. Its dual $\vsdual$ can be naturally identified with the real-valued functions on $S$, and then $\pospartofset{\vsdual}$ corresponds to $F$. The maps from $S$ into a vector lattice $\vl$ that are positive on $S$ correspond to positive linear maps from $\vs$ into $\vl$. The statements follow easily from these observations and an application of \cref{res:free_vector_lattice_over_pre-ordered_vector_space_as_lattice_of_functions} to $\povs$ with $\lsdual=\vsdual$.
	\end{proof}
	
For finite sets, \cref{res:free_vector_lattice_over_arbitrary_pointred_set} yields the following generalisation of the lattice-linear functional calculus.

\begin{corollary}\label{res:lattice_linear_calculus_generalised}
	Let $k,l\geq 0$ not both be zero, and let $\funlattice$ be the vector sublattice of $\Fun(\RR^k\times \RR_+^l)$ that is generated by the canonically defined coordinate functionals $t_1,\dotsc,t_k$ and $t_1^\prime,\dots,t_l^\prime$ on $\RR^k\times \RR_+^l$. Suppose that $\vl$ is a vector lattice, that $x_1,\dotsc,x_k\in\vl$, and that $x_1^\prime,\dotsc,x_l^\prime\in\vlpos$. Then there exists a unique vector lattice homomorphism $\psi\colon\funlattice\to\vl$ such that $\psi(t_i)=x_i$ for $i=1,\dotsc k$ and $\psi(t_i^\prime)=x_i^\prime$ for $i=1,\dotsc l$.
\end{corollary}

\begin{proof}
	Take a set $S$ with $k+l$ elements and a subset $\pospartofset{S}$ of it with $l$ elements, and apply \cref{res:free_vector_lattice_over_arbitrary_pointred_set}, where $F$ can now be identified with $\RR^k\times \RR_+^l$.
\end{proof}

\section{Free $\ct$-convex Banach lattices over pre-ordered Banach spaces}\label{sec:free_banach_lattices_over_pre-ordered_banach_spaces}

\noindent
In this section, we define the free $\ct$-convex Banach lattice $\FBLCpobs$ over a pre-ordered Banach space $\pobs$ and show that it exists. We also establish a number of its properties that can be derived from its definition. At the end of the section, we compare our categorical definition of $\FBLCpobs$ to that of the free Banach lattice over a Banach space in \cite{aviles_rodriguez_tradacete:2018}.

Let $\ct$ be a convexity type. In the general context of \cref{def:free_object}, we take for $\catone$ the pre-ordered Banach spaces as objects and the positive contractions as morphisms. Note that, under our standing assumption that $\RR$ be $\ct$-convex, $\ballbsdualpos$ consists of morphisms in $\catone$ for every pre-ordered Banach space $(\bs,{\bs^+})$. For $\cattwo$, we take the $\ct$-convex Banach lattices as objects and the contractive vector lattice homomorphisms as morphisms. The functor $\functor$ is defined by sending a Banach lattice $\bl$ to the pre-ordered Banach space $(\bl,\blpos)$ and viewing a contractive vector lattice homomorphism as a positive contraction. Simplifying the notation and terminology a little as was done in \cref{def:free_vector_lattice_over_a_pre-ordered_vector_space}, this leads to the following definition.

\begin{definition}\label{def:free_C-convex_banach_lattice_over_a_pre-ordered_banach_space}
	Let $\ct$ be a convexity type. A \emph{free $\ct$-convex Banach lattice over the pre-ordered Banach $\pobs$} is a pair $(j,\FBLCpobs)$, where
	$\FBLCpobs$ is a $\ct$-convex Banach lattice and $j\colon\bs\to\FBLCpobs$ is a positive contraction with the property that, for every positive contraction $\mor\colon\bs\to\bltwo$ into a $\ct$-convex Banach lattice $\bltwo$,  there exists a unique contractive vector lattice homomorphism $\facmor\colon\FBLCpobs\to\bltwo$ such that the diagram
	\begin{equation}\label{dia:free_C-convex_banach_lattice}
		\begin{tikzcd}
			\bs\arrow[r, "j"]\arrow[dr, "\mor", swap]& \FBLCpobs \arrow[d, "\facmor"]
			\\ & \vltwo
		\end{tikzcd}
	\end{equation}
	commutes.
\end{definition}

\begin{remark}\label{rem:comments_on_definition_of_free_banach_lattice}\quad
	\begin{enumerate}
	\item If $\FBLCpobs$ exists, then it is uniquely determined up to an isometric vector lattice isomorphism.  Indeed, for any two free objects, there are mutually inverse contractive vector lattice isomorphisms between them. These contractions must then, in fact, be isometries. Also by general principles, $j(\bs)$ generates $\FBLCpobs$ as a Banach lattice.
	\item A contraction from $\bs$ to a Banach lattice is positive on $\bspos$ if and only if it is positive on $\overline{\bspos}$. It is immediate from this that a free $\ct$-convex Banach lattice over $(\bs,\overline{\bspos})$ is also a free $\ct$-convex Banach lattice over $\pobs$, and vice versa.
	\end{enumerate}
\end{remark}

It will become clear below that $\FBLCpobs$ can be more or less routinely constructed from $\FVLpobs$, which we know to exist. Before proceeding to this construction, we point out two properties of $\FBLCpobs$ that are a priori clear, should it exist. None of these results is used in the construction as such, but they may serve as a motivation for it.

\begin{lemma}\label{res:norm_formula}
	Let $\ct$ be a convexity type, and suppose that $(j,\FBLCpobs)$ is a free $\ct$-convex Banach lattice over the pre-ordered Banach space $\pobs$.
	For $\freeelt\in\FBLCpobs$,
	\begin{equation}\label{eq:norm_formula}
	\norm{\freeelt}=\max_\mor\norm{\facmor(\freeelt)},
	\end{equation}
	where $\mor$ runs over the positive contractions $\mor\colon\bs\to\bltwo_\mor$ into $\ct$-convex Banach lattices $\bltwo_\mor$. In particular, for $x\in\bs$, we have
	\begin{equation}\label{eq:norm_formula_special}
		\norm{j(x)}=\max_\mor\norm{\mor(x)},
	\end{equation}
	where $\mor$ runs over the positive contractions $\mor\colon\bs\to\bltwo_\mor$ into $\ct$-convex Banach lattices $\bltwo_\mor$.
\end{lemma}

\begin{proof} For $\mor=j$, $\facmor$ is the identity operator, making clear that \eqref{eq:norm_formula} holds. The expression for $\norm{j(x)}$ is then clear.
\end{proof}

\begin{theorem}\label{res:FVL_embeds}
	Let $\ct$ be a convexity type, and suppose that $(j,\FBLCpobs)$ is a free $\ct$-convex Banach lattice over the pre-ordered Banach space $\pobs$.
	 Then:
	 \begin{enumerate}
	 	\item for $x\in\bs$, $j(x)\geq 0$ if and only if $x\in\overline{\bspos}$;
	 	\item $\Ker j=\overline{\bspos}\cap - \overline{\bspos}$, so that $j$ is injective if and only if $\overline{\bspos}$ is a cone.
	 \end{enumerate}
 Let $\vl$ be the dense vector sublattice of $\FBLCpobs$ that is generated by $j(\bs)$, and view $j$ as a positive linear map from $\bs$ into $E$. Then:
	\begin{enumerate}[resume]
		\item $(j,E)$ is a free vector lattice over the pre-ordered vector space $\pobsclosure$.
	\end{enumerate}
\end{theorem}

\begin{proof} Take a free vector lattice $(j_{\po},\FVLpobsclosure)$ over the pre-ordered vector space $\pobsclosure$.
	There is a unique vector lattice homomorphism $\fac{j}\colon\!\FVLpobsclosure\!\to\!\FBLCpobs$ such that the diagram
\begin{equation*}
	\begin{tikzcd}
		\bs\arrow[r, "j_{\po}"]\arrow[dr,"j", swap]& \FVLpobsclosure \arrow[d, "\fac{j}"]\\
		&\FBLCpobs
	\end{tikzcd}
\end{equation*}
is commutative. Hence $\fac{j}\big(\FVLpobsclosure\big)$ coincides with the dense vector sublattice of $\FBLCpobs$ that is generated by $j(\bs)$. We claim that $\fac{j}$ is injective. To see this, take a positive contraction $\mor\colon\bs\to\bltwo$ into a $\ct$-convex Banach lattice $\bltwo$. There exist a unique contractive vector lattice homomorphism $\facmor\colon\FBLCpobs\to\bltwo$ such that $\facmor\circ j=\mor$, and a unique vector lattice homomorphism $\fac{\mor_{\po}}\colon\FVLpobsclosure\to\vltwo$ such that $\fac{\mor_{\po}}\circ j_{\po}=\mor_{\po}$; here the subscript in $\mor_{\po}$ indicates again that we view $\mor$ as a positive linear map between the pre-ordered vector spaces $\pobsclosure$ and $(F,\pospartofset{F})$. For $x\in\bs$, we have
\begin{equation*}
	 \left(\facmor\circ\fac{j}\right)(j_{\po}(x))=\facmor(j(x))=\mor(x)=\mor_{\po}(x)=\fac{\mor_{\po}}(j_{\po}(x)).
\end{equation*}
Since they agree on a generating subset of $\FVLpobsclosure$, we have $\facmor\circ \fac{j}=\fac{\mor_{\po}}$.
As the $\facmor$ thus obtained separate the points of $\FBLCpobs$ by \cref{res:norm_formula}, we now see that, for $\freeelt\in\FVLpobsclosure$, $\fac{j}(\freeelt)=0$ if and only if $\fac{\mor_{\po}}(\freeelt)=0$ for all positive contractions $\mor$ from $\bs$ into $\ct$-convex Banach lattices, seen as positive linear maps with the pre-ordered vector space $\pobsclosure$ as domain. Thus, if $\fac{j}(\freeelt)=0$, then certainly $\fac{\mor_{\po}}(\freeelt)=0$ for all $\mor\in\ballbsdualpos$. By \cref{res:free_vector_lattice_over_pre-ordered_banach_space_as_lattice_of_functions}, this is already sufficient to imply that $\freeelt=0$. We conclude that $\fac{j}$ is injective, as claimed.

The statements in the theorem now follow easily from this and \cref{res:free_vector_lattice_over_pre-ordered_vector_space_as_lattice_of_functions}.
\end{proof}

We are not aware of a reference for the following lemma, which, together with \cref{res:convexity}, allows a short argument to show that $\FBLCpobs$ as constructed below is actually $\ct$-convex.  We include the easy proof.

\begin{lemma}\label{res:norming}
	Let $\bs$ be a normed space. Suppose that, for all $x$ in a dense subset of $\bs$,
	\begin{equation}\label{eq:norming}
		\norm{x}=\sup_\mor\norm{\mor(x)}
	\end{equation}
	where $\mor$ runs over a class of contractions $\mor\colon\bs\to Y_\mor$ into normed spaces $Y_\mor$. Then \eqref{eq:norming} holds for all $x\in\bs$.
\end{lemma}

\begin{proof}
	Take $x\in\bs$. It is clear that $\norm{x}\geq \sup_\mor\norm{\mor(x)}$. For the reverse inequality, take $\varepsilon>0$ and then an $x_0$ with $\norm{x-x_0}<\varepsilon/3$ such that $\norm{x_0}=\sup_\mor\norm{\mor(x_0)}$. There exists a $\mor_0$ in the class such that $\norm{x_0}\leq \norm{\mor_0 (x_0)}+\varepsilon/3$. Then
	\begin{align*}
		\norm{x}&\leq\norm{x-x_0}+\norm{x_0}\\
		&\leq\varepsilon/3+\norm{\mor_0(x_0)}+\varepsilon/3\\
		&\leq \norm{\mor_0(x_0-x)}+\norm{\mor_0(x)}+2\varepsilon/3\\
		&\leq \norm{\mor_0(x)}+\varepsilon\\
		&\leq \sup_\mor\norm{\mor(x)}+\varepsilon.
	\end{align*}
	As $\varepsilon$ was arbitrary, we are done.
\end{proof}

We now proceed with the construction of the free $\ct$-convex Banach lattice over a pre-ordered Banach space $\pobs$. It uses only the existence of $\FVLpobs$ from \cref{res:existence_of_free_vector_lattice}, and \cref{res:convexity,res:norming}.

\begin{proposition}
Let $\pobs$ be a pre-ordered Banach space, and let $\ct$ be a convexity type. Then there exists a free $\ct$-convex Banach lattice  $(j,\FBLCpobs)$ over $\pobs$.
\end{proposition}

\begin{proof}
We start by viewing $\pobs$ as a pre-ordered vector space and take the free vector lattice $(j_{\po},\FVLpo{\bs})$ over it, which we know to exist from \cref{res:existence_of_free_vector_lattice}. If $\mor\colon \bs\to\bltwo$ is a positive contraction into a $\ct$-convex Banach lattice $\bltwo$, then there exists a unique vector lattice homomorphism $\facmor\colon\FVLpo{\bs}\to\bltwo$ such that the diagram
\begin{equation*}
	\begin{tikzcd}
		\bs\arrow[r, "j_{\po}"]\arrow[dr, "\mor", swap]& \FVLpo{\bs} \arrow[d, "\facmor"]
		\\ & \bltwo
	\end{tikzcd}
\end{equation*}
commutes. Take $\xi \in\FVLpo{\bs}$. Motivated by \cref{res:norm_formula}, we set
\begin{equation}\label{eq:semi-norm}
	\rho(\xi )\coloneqq\sup_\mor\norm{\facmor(\xi )}
\end{equation}
where $\mor$ runs over the positive contractions $\mor\colon\bs\to\bltwo_\mor$ into $\ct$-convex Banach lattices $\bltwo_\mor$.
Evidently, $\rho(j_{\po}(x))=\sup_\mor\norm{\mor(x)}\leq\norm{x}$ for $x\in\bs$. Since the subset of $\FVLpo{\bs}$ where $\rho$ is finite is easily seen to be a vector sublattice, and since it contains its generating linear subspace $j_{\po}(\bs)$, $\rho$ is finite everywhere. It is a vector lattice semi-norm.
The kernel $\ker \rho(f)$ is an order ideal of $\FVLpo{\bs}$. Let $q\colon\FVLpo{\bs}\to\FVLpo{\bs}/\ker \rho$ denote the quotient vector lattice homomorphism, and introduce a vector lattice norm $\norm{\,\cdot\,}$ on $\FVLpo{\bs}/\ker \rho$ by setting $\norm{q(\xi )}\coloneqq\rho(\xi )$ for $\xi \in\FVLpo{\bs}$. Let $\FBLCpobs$ be the completion of $(\FVLpo{\bs}/\ker \rho,\norm{\,\cdot\,})$ and view $q$ as a map from $\FVLpo{\bs}$  into $\FBLCpobs$. We claim that, as the notation suggests, $(q\circ j_{\po},\FBLCpobs)$ is a free $\ct$-convex Banach lattice over $\pobs$.

There are a number of things to be checked. First of all, $q\circ j_{\po}$ is positive since $q$ and $j_{\po}$ are; it is contractive as $\norm{(q\circ j_{\po})(x)}=\rho(j_{\po}(x))=\sup_\mor\norm{\mor(x)}\leq\norm{x}$ for $x\in\bs$. Next, take a positive contraction $\mor\colon\bs\to\bltwo$ into a $\ct$-convex Banach lattice $\bltwo$. It follows from \eqref{eq:semi-norm} that $\ker \rho\subseteq\ker\facmor$, so there exists a unique vector lattice homomorphism $\fac{\facmor}\colon\FVLpo{\bs}/\ker \rho\to\bltwo $ such that $\fac{\facmor}\circ q=\facmor$. For $\xi \in\FVLpo{\bs}$, we have $\norm{\fac{\facmor}(q(\xi ))}=\norm{\facmor(\xi )}\leq\rho(\xi )=\norm{q(\xi )}$. Hence $\fac{\facmor}$ is a contractive vector lattice homomorphism. It extends to a contractive vector lattice homomorphism $\fac{\fac{\facmor}}\colon\FBLCpobs\to\bltwo$. It is a consequence of the construction that $\fac{\fac{\facmor}}\circ(q\circ j_{\po})=\mor$. Furthermore, since $j_{\po}(\bs)$ generates $\FVLpo{\bs}$ as a vector lattice, $(q\circ j_{\po})(\bs)$ generates $\FBLCpobs$ as a Banach lattice. Hence $\fac{\fac{\facmor}}$ is the only contractive vector lattice homomorphism (by automatic continuity: even the only vector lattice homomorphism) such that $\fac{\fac{\facmor}}\circ(q\circ j_{\po})=\mor$.

The construction is summarised in the following commutative diagram:
\begin{equation*}
	\begin{tikzcd}
		\bs\arrow[r, "j_{\po}"]\arrow[ddr,"\mor"]& \FVLpo{\bs}\arrow[r,"q"]\arrow[dd,"\facmor"]&\FVLpo{\bs}/\ker\rho_a\arrow[ddl,"\fac{\facmor}"]\arrow[r,hookrightarrow]&\FBLCpobs\arrow[ddll,"\fac{\fac{\facmor}}"]\\
		\\
		&\bltwo&&
	\end{tikzcd}
\end{equation*}

It remains to verify that $\FBLCpobs$ is $\ct$-convex. For this, take $\xi \in\FVLpo{\bs}$. Then
\[
\norm{q(\xi )}=\rho(\xi )=\sup_\mor\norm{\facmor(\xi )}=\sup_\mor\Bignorm{\fac{\facmor}(q(\xi ))}=\sup_\mor\Bignorm{\fac{\fac{\facmor}}(q(\xi ))}.
\]
Since $q(\FVLpo{\bs})$ is dense in $\FBLCpobs$, \cref{res:norming} yields that
\[
\norm{\freeelt}=\sup_\mor\Bignorm{\fac{\fac{\facmor}}(\freeelt)}
\]
for $\freeelt\in\FBLCpobs$. By \cref{res:convexity}, $\FBLCpobs$ is $\ct$-convex since all codomains of the $\fac{\fac{\facmor}}$ are.
\end{proof}

We shall now establish some further properties of $(j,\FBLCpobs)$. With the exception of \cref{ex:not_an_isometry}, they are all essentially direct consequences of its universal property, just as \cref{res:norm_formula} and \cref{res:FVL_embeds}. The concrete construction of it in the proof of \cref{res:existence_of_free_vector_lattice} plays no role.

A positive contractive $\mor$ in diagram \eqref{dia:free_C-convex_banach_lattice} has a factoring vector lattice homomorphism $\facmor$ that is from the definitions only known to be contractive. However, more is true.

 \begin{lemma}\label{res:equality_of_norms}
 	Let $\ct$ be a convexity type, and let $(j,\FBLCpobs)$ be a free $\ct$-convex Banach lattice over the pre-ordered Banach space $\pobs$. If $\mor\colon\bs\to\bltwo$ is a positive contraction into a $\ct$-convex Banach lattice, then $\norm{\mor}=\norm{\facmor}$.
 	Furthermore, if $T\colon\bs\to\vltwo$ is an operator, then there exists a unique vector lattice homomorphism $\fac{T}\colon\FBLCpobs\to\vltwo$ making diagram \eqref{dia:free_C-convex_banach_lattice} commutative. Moreover, $\norm{\fac{T}}=\norm{T}$.
 \end{lemma}

 \begin{proof}
 	We may suppose that $\mor\neq 0$. As $\mor=\facmor\circ j$, we have $\norm{\mor}\leq\norm{\facmor}$. For the reverse inequality, we observe that $\facmor$ is even the unique vector lattice homomorphism that makes diagram \eqref{dia:free_C-convex_banach_lattice} commutative. To see this, we note that any such vector lattice homomorphism $\facmor$ is uniquely determined on $j(\bs)$ and then also on the vector sublattice generated by $j(\bs)$. As this is dense and $\facmor$ is automatically continuous, $\facmor$ is indeed unique.\footnote{This argument is also used in the proof of  \cite[Corollary~3.5]{jardon-sanchez_laustsen_taylor_tradacete_troitsky:2022}.} Now $\mor/\norm{\mor}$ is also a positive contraction, so it has an accompanying contractive vector lattice homomorphism $\fac{(\mor/\norm{\mor})}$ such that $\mor/\norm{\mor}=\fac{(\mor/\norm{\mor})}\circ j$. Since the vector lattice homomorphism $\facmor/\norm{\mor}$ also satisfies $\mor/\norm{\mor}=\left(\facmor/\norm{\mor}\right)\circ j$, we conclude that $\fac{(\mor/\norm{\mor})}=\facmor/\norm{\mor}$. As $\fac{(\mor/\norm{\mor})}$ is contractive, we have $\norm{\facmor}\leq\norm{\mor}$, as desired. The statements for general positive $T$ follow by scaling from the uniqueness argument above and the equality of norms in the contractive case.
 \end{proof}

We now concentrate on the properties of the positive contraction $j$. Its being the zero map or not is obviously correspond to  $\FBLCpobs$ being the zero lattice or not, but one can say a little more.

\begin{proposition}\label{res:degeneracy} Let $\ct$ be a convexity type, and let $(j,\FBLCpobs)$ be a free $\ct$-convex Banach lattice over the pre-ordered Banach space $\pobs$. Then the following are equivalent:
		\begin{enumerate}
			\item $\FBLCpobs=\zerospace$;
			\item $\ballbsdualpos= \zerospace$;
			\item $\overline{\bspos}= \bs$;
			\item $j=0$.
		\end{enumerate}
	If $\FBLCpobs\neq \zerospace$, then $\norm{j}=1$.
\end{proposition}

\begin{proof} If $\FBLCpobs=\zerospace$, then $\ballbsdualpos=\zerospace$ by diagram \eqref{dia:free_C-convex_banach_lattice}. The latter implies that  $\overline{\bspos}=\bs$ by \cref{res:free_vector_lattice_over_pre-ordered_banach_space_as_lattice_of_functions}. If $\overline{\bspos}=\bs$, then clearly $j=0$. If $j=0$, then $\FBLCpobs=\zerospace$ as it is generated by $j(\bs)$.
	
	 Suppose that $\FBLCpobs\neq\zerospace $. By what has already been established, there exists a non-zero $x^*\in\ballbsdualpos$. Using \cref{res:equality_of_norms}, we see that $\norm{x^*}=\norm{\fac{x^*}\circ j}\leq\norm{\fac{x^*}}\norm{j}=\norm{x^*}\norm{j}$. Hence $\norm{j}\geq 1$, so that $\norm{j}=1$.
\end{proof}

It is clear from \cref{res:FVL_embeds} that $j$ is injective if and only if $\overline{\bspos}$ is a cone, that it is bipositive if and only if $\bspos$ is closed, and that it is a bipositive injection if and only if $\bspos$ is a closed cone. The next result clarifies when it is an embedding of $\bs$ as a pre-ordered Banach space.
We recall that the positive cone $\bspos$ of a pre-ordered Banach $\pobs$ space is \emph{normal} if there exists a $c>0$ such that $\norm{x}\leq c\norm{y}$ for all $x,y\in\bs$ such that $0\leq x\leq y$.

\begin{theorem}\label{res:embedding_as_an_ordered_banach_space} Let $\ct$ be a convexity type, and let $(j,\FBLCpobs)$ be a free $\ct$-convex Banach lattice over the pre-ordered Banach space $\pobs$. Then $j$ is a bipositive injection such that $j(\bs)$ is closed if and only if $\bspos$ is a closed normal cone.
	
	When this is the case, there exists a $K>0$ such that, for each $x^*$ in $\bsdual$, there exist $x_1^*,x_2^*\in \bsdualpos$ with $x^*=x_1^*-x_2^*$  and $\norm{x_1^*}+\norm{x_2^*}\leq K\norm{x^*}$. For any such $K$   we have $1/K\norm{x}\leq\norm{j(x)}\leq\norm{x}$ for $x\in\bs$.
\end{theorem}

\begin{proof}
For statement in the first paragraph, it is, in view of the remarks preceding the theorem,  sufficient to suppose that $\bspos$ is a closed cone and then, knowing that $j$ is a bipositive injection, show that $j(\bs)$ being closed is equivalent to the normality of $\bspos$. Suppose that $j(\bs)$ is closed. Then $x\mapsto\norm{j(x)}$ is an equivalent norm on $\bs$, in which case the positivity of $j$ and the normality of the cone of the Banach lattice $\FBLCpobs$ are easily seen to imply the normality of $\bspos$. Conversely, suppose that $\bspos$ is normal. Then $\bsdualpos$ is generating in $\bsdual$ by \cite[Theorem~2.40]{aliprantis_tourky_CONES_AND_DUALITY:2007}. It follows from this that there exists a $K>0$ such that, for each $x^*$ in $\bsdual$, there exist $x_1^*,x_2^*\in \bsdualpos$ with $x^*=x_1^*-x_2^*$ and $\norm{x_1^*}+\norm{x_2^*}\leq K\norm{x^*}$;  see \cite[Theorem~2.37]{aliprantis_tourky_CONES_AND_DUALITY:2007} or, more generally, \cite[Theorem~4.1]{de_jeu_messerschmidt:2014}. Take $x\in\bs$ and then $x^*\in\ballbsdual$ such that $\abs{x^*(x)}=\norm{x}$. Choose $x_1^*,x_2^*\in \bsdualpos$ such that $x^*=x^*_1-x^*_2$ and $\norm{x_1^*}+ \norm{x_2^*}\leq K$. With $\fac{x_1^*}$, $\fac{x_2^*}$ as in \cref{res:equality_of_norms}, we then see that
\begin{equation*}
	\begin{split}
	\norm{x}&=\abs{x^*_1(x)-x^*_2(x)}\\
	&=\abs{\fac{x^*_1}(j(x))-\fac{x^*_2}(j(x))}\\
	&\leq \left(\norm{\fac{x^*_1}}+\norm{\fac{x^*_2}}\right)\norm{j(x)}\\
	&=\left(\norm{x^*_1}+\norm{x^*_2}\right)\norm{j(x)}\\
	&\leq K \norm{j(x)}.
	\end{split}
\end{equation*}
Hence $j$ is bounded below, so that $j(\bs)$ is closed.

The statements in the second paragraph have now also been established.
\end{proof}

When $\bs$ is a Banach lattice, $j$ may be a vector lattice homomorphism or not. This happens only in two trivial cases.

\begin{proposition}\label{res:image_sublattice}
	Let $\ct$ be a convexity type, and let $(j,\FBLCpobs)$ be a free $\ct$-convex Banach lattice over the pre-ordered Banach space $\pobs$, where $\bs$ is a Banach lattice.  Then the following are equivalent:
	\begin{enumerate}
		\item\label{part:image_sublattice_1} $j$ is a vector lattice homomorphism;
		\item\label{part:image_sublattice_2} $j(\bs)$ is a Banach sublattice of $\FBLCpobs$;
		\item\label{part:image_sublattice_3} $\bs$ is zero- or one-dimensional;
		\item\label{part:image_sublattice_4} $(\id,\bs)$ is a free $\ct$-convex Banach lattice over $\pobs$, where $\id$ is the identity operator on $\bs$.
	\end{enumerate}
\end{proposition}

\begin{proof}
	The equivalence of the parts~\ref{part:image_sublattice_1} and~\ref{part:image_sublattice_2} follows from \cref{res:embedding_as_an_ordered_banach_space}. 
	
	Suppose that $j$ is a vector lattice homomorphism. Then every positive contraction $\mor\colon \bs\to\bltwo$ into a Banach lattice is a vector lattice homomorphism; the same is then true for every positive operator. We claim that this implies that $\bs$ is zero- or one-dimensional. To see this, we may suppose that $\bs\neq\zerospace$. Take a non-zero $x\in\bs$ and then an $x^*\in\bsdualpos$ such that $x^*(x)\neq 0$. As $x^*(x^+)\wedge x^*(x^-)=0$, either $x^*(x^+)>0$ and $x^*(x^-)=0$, or $x^*(x^+)=0$ and $x^*(x^-)>0$. In the first case, take $y^*\in \bsdualpos$. Then  $(x^*+y^*)(x^+)\wedge (x^*+y^*)(x^-)=0$. As $(x^*+y^*)(x^+)\geq x^*(x^+)>0$, we have $(x^*+y^*)(x^-)=0$. Hence $y^*(x^-)=0$. Since $y^*$ was arbitrary, we conclude that $x^-$=0. In the second case, it follows similarly that $x^+=0$. We thus see that $\bs$ is linearly ordered, so that it is one-dimensional. Hence part~\ref{part:image_sublattice_1} implies part~\ref{part:image_sublattice_3}. It is easy to see that part~\ref{part:image_sublattice_3} implies part~\ref{part:image_sublattice_4}, and clearly the latter implies part~\ref{part:image_sublattice_1}.
\end{proof}

If $\bspos=\zerospace$, then $\bsdualpos=\bsdual$, so that one can take $K=1$ in \cref{res:embedding_as_an_ordered_banach_space}. This gives the following.

\begin{corollary}\label{res:isometry_when_zero_cone} Let $\ct$ be a convexity type, and let $(j,\FBL^\ct[(\bs,\zerospace)])$ be a free $\ct$-convex Banach lattice over the pre-ordered Banach space $(\bs,\zerospace)$. Then $j$ is an isometry.
\end{corollary}

It appears to be a non-trivial open question to give a general convenient criterion for $j$ to be isometric when $\bspos\neq\zerospace$.
There is, however, one natural situation where this is easily seen to be the case. Indeed, if $\bs$ is a Banach lattice that is itself $\ct$-convex, then the identity operator on $\bs$ occurs in \eqref{eq:norm_formula_special}, implying that $j$ is isometric. In general, we have the following pointwise results.

\begin{lemma}\label{res:norm_of_image}
Let $\ct$ be a convexity type, and let $(j,\FBLCpobs)$ be a free $\ct$-convex Banach lattice over the pre-ordered Banach space $\pobs$. Then:
\begin{enumerate}
	\item\label{part:norm_of_image_1} $\norm{j(x)}\geq \sup_{\mor\in\ballbsdualpos}\abs{\mor(x)}$ for $x\in \bs$;
	\item\label{part:norm_of_image_2} $\norm{j(x)}=\max_{\mor\in\ballbsdualpos}\mor(x)$ for $x\in \bspos$.
\end{enumerate}
\end{lemma}

\begin{proof}
	Take $x\in\bs$. Then $\abs{\mor(x)}=\abs{(\facmor\circ j)(x)}\leq\norm{j(x)}$ for $\mor\in\ballbsdualpos$, establishing the first part. For the second part, take $x\in\bspos$ and then a positive contractive $\psi$ in the dual of $\FBLCpobs$ such that $(\psi\circ j)(x)=\psi(j(x))=\norm{j(x)}$, which is possible as $j(x)$ is positive. The second part follows from this and the first part.
\end{proof}

Supplementing \cref{res:image_sublattice}, we collect a number of results when $\bs$ is a Banach lattice in the following.

\begin{theorem}\label{res:summary_for_banach_lattices}
	Let $\ct$ be a convexity type, and let $(j,\FBLCpobs)$ be a free $\ct$-convex Banach lattice over the pre-ordered Banach space $\pobs$, where $\bs$ is a Banach lattice. Then:
	\begin{enumerate}
		\item\label{part:summary_for_banach_lattices_1} $j$ is a bipositive embedding such that
		\begin{enumerate}
			\item\label{part:summary_for_banach_lattices_1_1} $1/2 \norm{x}\leq\norm{j(x)}\leq\norm{x}$ for $x\in \bs$;
			\item\label{part:summary_for_banach_lattices_1_2} $\norm{j(x)}=\norm{x}$ for $x\in\bspos\cup -\bspos$;
		\end{enumerate}
		\item\label{part:summary_for_banach_lattices_2} $j$ is isometric when $\bs$ is $\ct$-convex, in which case   $(j\circ\fac{\id})\colon\FBLCpobs\to\FBLCpobs$, where $\id\colon\bs\to\bs$ is the identity map, is a positive projection onto $j(\bs)$ that has norm one when $\bs\neq\zerospace$;		
	\end{enumerate}
\end{theorem}

\begin{proof}
	Part~\ref{part:summary_for_banach_lattices_1_1} follows from \cref{res:embedding_as_an_ordered_banach_space}, where one can take $K=2$.
	
	Part~\ref{part:summary_for_banach_lattices_1_2} follows from part~\ref{part:norm_of_image_2} of \cref{res:norm_of_image}.
	
	In part~\ref{part:summary_for_banach_lattices_2}, the fact that $j$ is isometric when $\bs$ is $\ct$-convex was observed prior to \cref{res:norm_of_image}. The statement that $j\circ\fac{\id}$ is a projection onto $j(\bs)$ is the analogue of \cite[Corollary~2.6]{aviles_rodriguez_tradacete:2018}; it is clearly positive.  Since it is contractive, it has norm one when it is non-zero.	
\end{proof}

Thus, perhaps somewhat remarkably, when the Banach lattice $\bs$ is $\ct$-convex but not zero- or one-dimensional,  $\bs$ sits complementedly and isometrically inside $\FBLCpobs$ as an ordered Banach space, but not as a vector sublattice.

If $\pobs$ is a general pre-ordered Banach space, then, by  \cref{res:isometry_when_zero_cone}, it is sufficient that $\bspos=\zerospace$ for $j$ to be isometric. This is, however, not necessary, as is shown by \cref{res:summary_for_banach_lattices}. By \cref{res:norm_of_image}, for $j$ to be isometric it is necessary that $\ballbsdualpos$ be norming for $\bspos$, but this is not sufficient. Indeed, as the following example shows, even for Banach lattices, where it \emph{is} norming, the bipositive embedding $j$ need \emph{not} be isometric. It also shows that the factor $1/2$ in \cref{res:summary_for_banach_lattices} is sharp.

\begin{example}\label{ex:not_an_isometry}
	Take a convexity type $\ct$ such that the $\ct$-convex Banach lattices are the $\infty$-convex Banach lattices with $\infty$-constant 1. We write $(j,\FBLinfpobs)$ for the free $\ct$-convex Banach lattice over the pre-ordered Banach space $\pobs$. We shall show in \cref{subsec:case_infty} that
	\[
	\norm{\freeelt}=\sup_{\mor\in\ballbsdualpos}\abs{\facmor(\freeelt)}
	\]
	for $\freeelt\in\FBLinfpobs$. Accepting this for the moment, we thus have that
	\[
	\norm{j(x)}=\sup_{\mor\in\ballbsdualpos}\abs{\mor(x)}
	\]
	for $x\in\bs$. For $\bs$, we take the Banach lattice $\ell_1^2$. The positive contractive functionals on $\bs$ are the $\mor_{\alpha_1,\alpha_2}$ for $0\leq\alpha_1,\alpha_2\leq 1$, defined by setting $\mor_{\alpha_1,\alpha_2}(x)\coloneqq \alpha_1 x_1+\alpha_2 x_2$ for $(x_1,x_2)\in\bs$.
	It is then easy to see that, for  $x=(x_1,x_2)\in\bs$,
	\begin{equation*}
		\norm{j(x)}=
		\begin{cases}
			\norm{x}&\text{ when }x\in\bspos\cup-\bspos;\\
			\max(\abs{x_1},\abs{x_2})&\text{otherwise}.
		\end{cases}
	\end{equation*}
On taking $x=(1,-1)$, we see that $j$ is not isometric and that the factor $1/2$ in \cref{res:summary_for_banach_lattices} is sharp.
\end{example}

\begin{remark}
In view of \cref{res:summary_for_banach_lattices}, \cref{ex:not_an_isometry} can only work because $\ell_1^2$ is not $\infty$-convex with $\infty$-convexity constant 1. It is open whether, for a Banach lattice $\bs$, the fact that $j\colon\bs\to\FBLCpobs$ is an isometry implies that $\bs$ is $\ct$-convex.
\end{remark}

To conclude, we consider the question whether $\FBLCpobs$ can be realised as a function lattice. As point evaluations are then automatically continuous, a moment's thought shows that this is equivalent to requiring that the $\facmor$ for $\mor\in\ballbsdualpos$ separate its points; this was already observed in \cite[Remark~6.2]{jardon-sanchez_laustsen_taylor_tradacete_troitsky:2022}. This motivates the following.

For $\freeelt\in\FBLCpobs$, define $\widehat\freeelt\colon\ballbsdualpos\to\RR$ by setting
\begin{equation}\label{eq:transport}
	\widehat\freeelt(\mor)\coloneqq \facmor(\freeelt)
\end{equation}
for $\mor\in\ballbsdualpos$. The map $\freeelt\mapsto\widehat\freeelt$ is then a vector lattice homomorphism from $\FBLCpobs$ into the vector lattice of all positively homogeneous functions on $\ballbsdualpos$, and such that $\widehat{j(x)}=\delta_x$ for $x\in\bs$. It follows from \eqref{eq:norm_formula} that the convergence of a sequence in $\FBLCpobs$ implies the uniform convergence of its image. As the dense vector sublattice of $\FBLCpobs$ that is generated by $j(\bs)$ clearly maps into $\contballposh$, this is true for $\FBLCpobs$ itself.

\begin{proposition}\label{res:injectivity_of_canonical_map}
	Let $\ct$ be a convexity type, and let $(j,\FBLCpobs)$ be a free $\ct$-convex Banach lattice over the pre-ordered Banach space $\pobs$. Then the following are equivalent:
	\begin{enumerate}
		\item\label{part:injectivity_of_canonical_map_1}
		the \hskip -.5pt vector lattice homomorphism \hskip -.5 pt$\freeelt\!\mapsto\!\widehat\freeelt$ in \hskip-1pt \eqref{eq:transport} from $\FBLCpobs$ into \hskip -.5pt $\contballposh$ is injective;
		\item\label{part:injectivity_of_canonical_map_2} the contractive vector lattice homomorphisms $\facmor\colon\FBLCpobs\to\RR$ for $\mor\in\ballbsdualpos$ separate the points of $\FBLCpobs$;
		\item\label{part:injectivity_of_canonical_map_3} the contractive vector lattice homomorphisms $\facmor\colon\FBLCpobs\to\bltwo_\mor$ separate the points of $\FBLCpobs$, as $\mor$ runs over all positive contractions $\mor\colon\bs\to\bltwo_\mor$ into finite dimensional  $\ct$-convex Banach lattices $\bltwo_\mor$.
	\end{enumerate}
\end{proposition}

\begin{proof}
	 As observed earlier, it is clear that the parts~\ref{part:injectivity_of_canonical_map_1} and~\ref{part:injectivity_of_canonical_map_2} are equivalent. Obviously, part~\ref{part:injectivity_of_canonical_map_2} implies part~\ref{part:injectivity_of_canonical_map_3}. We show that part~\ref{part:injectivity_of_canonical_map_3} implies part~\ref{part:injectivity_of_canonical_map_1}. Suppose that $\freeelt\in\FBLCpobs$ and that $\widehat\freeelt=0$. Take a $\mor$ as in part~\ref{part:injectivity_of_canonical_map_3}. Since $\bltwo_\mor$ can be identified as a vector lattice with $\RR^n$ for some $n$, there are, after such an identification, $\mor_1,\ldots, \mor_n\in\ballbsdualpos$ and $\lambda_1,\dotsc, \lambda_n\geq 0$ such that $\lambda_1 \mor_1,\cdots,\lambda_n \mor_n$ are the coordinate components of $\mor$. By the uniqueness of the factoring morphism, the coordinate components of $\facmor$ are  $\lambda_1 \fac{\mor_1},\cdots,\lambda_n \fac{\mor_n}$. Hence $\facmor(\freeelt)=0$. As $\mor$ was arbitrary, we have $\freeelt=0$.
\end{proof}

\begin{remark}\label{rem:comparison_with_JFA-definition}
With the results in the current section available, a meaningful comparison of the definitions of free Banach lattices in \cite[p.2956]{aviles_rodriguez_tradacete:2018} and in our  \cref{def:free_C-convex_banach_lattice_over_a_pre-ordered_banach_space} is now possible.\footnote{There appears to be no definition as such of what are called  free Banach lattices under convexity conditions in \cite{jardon-sanchez_laustsen_taylor_tradacete_troitsky:2022}.} Let $\bs$ be a Banach space.

In \cite{aviles_rodriguez_tradacete:2018}, a free Banach lattice over $\bs$ is a pair $(\phi_\bs,\FBL[\bs])$, where $\FBL[\bs]$ is a Banach lattice and $\phi_\bs\colon\bs\to\FBL[\bs]$ is an isometry with the property that, for every operator $T\colon\bs\to\bltwo$ into a Banach lattice $\bltwo$, there exists a unique vector lattice homomorphism $\fac{T}\colon\FBL[\bs]\to\bltwo$ with $\norm{\fac{T}}=\norm{T}$ making the diagram
\begin{equation*}\label{dia:free_banach_lattice}
	\begin{tikzcd}
		\bs\arrow[r, "\phi_E"]\arrow[dr, "T", swap]&  \FBL[\bs] \arrow[d, "\fac{T}"]
		\\ & \bltwo
	\end{tikzcd}
\end{equation*}
commutative. Via scaling it is obviously equivalent to require this for all contractions $T$.

According to \cref{def:free_C-convex_banach_lattice_over_a_pre-ordered_banach_space},  a free $\emptyset$-convex Banach lattice over the pre-ordered Banach $(\bs,\zerospace)$ is a pair $(j,\FBL^\emptyset[(\bs,\zerospace)])$, where
$\FBL^\emptyset[(\bs,\zerospace)]$ is a Banach lattice and $j\colon\bs\to\FBL^\emptyset[(\bs,\zerospace)]$ is a contraction with the property that, for every contraction $\mor\colon\bs\to\bltwo$ into a Banach lattice $\bltwo$,  there exists a unique contractive vector lattice homomorphism $\facmor\colon\FBL^\emptyset[(\bs,\zerospace)]\to\bltwo$ making the diagram
\begin{equation}\label{dia:free_empty_set-convex_banach_lattice}
	\begin{tikzcd}
		\bs\arrow[r, "j"]\arrow[dr, "\mor", swap]& \FBL^\emptyset[(\bs,\zerospace)] \arrow[d, "\facmor"]
		\\ & \vltwo
	\end{tikzcd}
\end{equation}
commutative. We know from \cref{res:isometry_when_zero_cone} and \cref{res:equality_of_norms}, however, that it is automatic that $j$ in diagram~\eqref{dia:free_empty_set-convex_banach_lattice} is an isometry, and that $\norm{\mor}=\norm{\facmor}$. Hence $\FBL^\emptyset[(\bs,\zerospace)]$ and $\FBL[\bs]$ coincide.

We thus see that there is some redundancy in \cite{aviles_rodriguez_tradacete:2018},
where the definition of a free Banach lattice is intuitive rather than motivated from category theory. The redundancy is harmless, but it would be wrong to adapt the obvious analogous intuitive definition also in the general case of a free $\ct$-convex Banach lattice over a pre-ordered Banach space $\pobs$. This would entail that, in diagram~\eqref{dia:free_C-convex_banach_lattice}, it is required that $j$ be a positive isometry, and that $\norm{\mor}=\norm{\facmor}$. Clearly, if it exists, this alternate free object coincides with our $\FBLCpobs$, which we know to exist and of which we know a few basic properties. It is always automatic that $\norm{\facmor}=\norm{\mor}$ by \cref{res:equality_of_norms}, and $j$ is still automatically isometric when $\bspos=\zerospace$ by \cref{res:isometry_when_zero_cone}. The difficulties arise (or, rather, new phenomena occur) when $\bspos\neq \zerospace$. Then the free object can fail to exist when the alternate definition is used. Obviously, it then fails to exist when $\FBLCpobs=\zerospace$ as in \cref{res:degeneracy} but $\bs\neq\zerospace$. Slightly more interestingly, \cref{res:FVL_embeds} shows that it fails to exist when $\overline{\bspos}$ is not a cone because then $j$ is also not even injective. Much more interestingly,  \cref{res:summary_for_banach_lattices} and \cref{ex:not_an_isometry} show that the free object can then even fail to exist in cases where $\bs$ is a Banach lattice and the map $j$ in \cref{def:free_C-convex_banach_lattice_over_a_pre-ordered_banach_space} is a bipositive contraction with closed range, which happens to be non-isometric.

It appears that is preferable to follow the general category theoretical definition of a free object that led to   \cref{def:free_C-convex_banach_lattice_over_a_pre-ordered_banach_space}.
\end{remark}

\section{Realisation of the free $p$-convex Banach lattice with $p$-convexity constant 1 over a pre-ordered Banach space as a function lattice}\label{sec:realisation}

\noindent Let $1\leq p\leq\infty$ and let $\pobs$ be a pre-ordered Banach space. Take a convexity type $\ct_p$ such that the $\ct_p$-convex Banach lattices are the $p$-convex Banach lattices with $p$-convexity constant 1. We know from \cref{sec:free_banach_lattices_over_pre-ordered_banach_spaces}
that $(j,\FBL^{{\ct}_p}[\pobs])$, the free $p$-convex Banach lattice with $p$-convexity constant 1 over $\pobs$ exists.  We denote it by $(j,\FBLppobs)$; the free Banach lattice $\FBLp[\bs]$ as in \cite{jardon-sanchez_laustsen_taylor_tradacete_troitsky:2022} is then our $\FBLp[\bs,\zerospace ]$. In this section, we shall give a realisation of $\FBLppobs$ as a Banach lattice of functions. For $\FBLp[\bs,\zerospace ]$, such a realisation is known: for $p=1$ this can be found in \cite[Theorem~2.5]{aviles_rodriguez_tradacete:2018}, for $1<p<\infty$ in \cite[Theorem~6.1]{jardon-sanchez_laustsen_taylor_tradacete_troitsky:2022}, and for $p=\infty$ in \cite[Theorem~5.3]{jardon-sanchez_laustsen_taylor_tradacete_troitsky:2022}. When looking at the pertinent realisations as function lattices on $\ballbsdual$, there is a natural approach to find realisations as function lattices on $\ballbsdualpos$ in the general case. As we shall see, this approach is successful.

Although our proof of this fact benefits from some of the ingenious arguments in the proof of \cite[Theorem~~6.1]{jardon-sanchez_laustsen_taylor_tradacete_troitsky:2022}, our method is different from that in \cite{jardon-sanchez_laustsen_taylor_tradacete_troitsky:2022} or \cite{aviles_rodriguez_tradacete:2018}. In the proof of \cite[Theorem~~6.1]{jardon-sanchez_laustsen_taylor_tradacete_troitsky:2022}, the existence of $\FBLp[\bs,\zerospace ]$, which is at that stage already known from \cite[Theorem~~3.3]{jardon-sanchez_laustsen_taylor_tradacete_troitsky:2022}, is established once more by showing that a concrete function lattice has the desired universal property. The existence of  $\FBL^{(1)}[(\bs,\zerospace)])$ in the proof of \cite[Theorem~2.5]{aviles_rodriguez_tradacete:2018} is similarly intertwined with its realisation as a function lattice. In the current section, however, we start with $\FBLppobs$, which we already have shown to exist, essentially by building on an existence result from universal algebra. We then analyse its norm to the extent where it becomes clear that it can be realised as a function lattice on $\ballbsdualpos$ as in \cref{res:injectivity_of_canonical_map}. We shall make additional comments on the difference in \cref{rem:completeness_not_needed_infinity,rem:completeness_not_needed_p}.

\begin{remark}
	Before proceeding, the following should be noted, covering a more general situation than positive contractions and $p$-convexity constant 1, as in \cite[Theorem~ 6.1]{jardon-sanchez_laustsen_taylor_tradacete_troitsky:2022}): when $1\leq p<\infty$, $\pobs$ is a pre-ordered Banach space, and $T\colon\bs\to\bltwo$ is a positive operator into a $p$-convex Banach lattice $\bltwo$ with $p$-convexity constant $M^{(p)}(\bltwo)$, there exists a unique vector lattice homomorphism $\fac{T}\colon\FBLppobs\to\bltwo$ such that $T\circ j=\mor$. Moreover,  $\norm{\fac{T}}\leq M^{(p)}(\bltwo)\norm{T}$. This follows from \cref{res:equality_of_norms} and the values of the constants in \cref{rem:renorming} when renorming $\bltwo$ to have $p$-convexity constant 1, after which the universal property of $\FBLppobs$ can be used.
\end{remark}

In our analysis of the norm of $\FBLppobs$, we shall distinguish between finite and infinite $p$. The following preparatory result will be needed in both cases.

\begin{lemma}\label{res:composition}
	Let $\ct$ be a convexity type, and let $(j,\FBLCpobs)$ be a free $\ct$-convex Banach lattice over the pre-ordered Banach space $\pobs$.
 Let $\bltwo$ and $\bltwo^\prime$ be $\ct$-convex Banach lattices, let $\mor\colon\bs\to\bltwo$ be a positive contraction, and let $\psi\colon\bltwo\to\bltwo^\prime$ be a contractive vector lattice homomorphism. Then $\fac{\psi\circ\mor}=\psi\circ\facmor$.
\end{lemma}

\begin{proof}
	We have $(\psi\circ\facmor)\circ j=\psi\circ\mor$. As $\psi\circ\facmor$ is a contractive vector lattice homomorphism, it must be $\fac{\psi\circ\mor}$.
\end{proof}

\subsection{The case $p=\infty$}\label{subsec:case_infty}

This is the easy case. Take $\freeelt\in\FBLinfpobs$. We know from  \cref{res:norm_formula} that
\begin{equation*}
	\norm{\freeelt}=\max_\mor\norm{\facmor(\freeelt)},
\end{equation*}
where $\mor$ runs over all positive contractions $\mor\colon\bs\to\bltwo_\mor$ into $\infty$-convex Banach lattice with $\infty$-convexity constant 1. Take such a $\mor$. By \cref{res:characterisation_infinite_p}, there exists a contractive vector lattice homomorphism $\psi\colon\bltwo_\mor\to\RR$ such that $\norm{\facmor(\freeelt)}=\abs{\psi(\facmor(\freeelt))}$. \cref{res:composition} implies that $\norm{\facmor(\freeelt)}=\abs{\fac{\psi\circ\mor}(\freeelt)}$. It is now clear that
\begin{equation}\label{eq:iso_AM-spaces}
	\norm{\freeelt}=\max_{x^*\in\ballbsdualpos}\abs{\fac{x^*}(\freeelt)}.
\end{equation}

Using \eqref{eq:iso_AM-spaces}, we can now give a realisation of $\FBLinfpobs$ as a vector sublattice of $\contballposh$ as in \cref{res:injectivity_of_canonical_map}. For $\freeelt\in\FBLinfpobs$, define $\widehat{\freeelt}\colon\ballbsdualpos\to\RR$ by setting
\[
\widehat{\freeelt}(x^*)=\fac{x^*}(\freeelt).
\]

Then $\widehat{j(x)}=\delta_x$ for $x\in\bs$. Recalling that $\FBLinfpobs$ is generated by $j(\bs)$, we have the following result for free AM-spaces.

\begin{theorem}\label{res:AM-realisation}
	Let $\pobs$ be a pre-ordered Banach space. Supply $\contballposh$ with the maximum norm, and let $\bl$ be the closed vector sublattice of the Banach lattice $\contballposh$ that is generated by $\{\delta_x:x\in\bs\}$. Define $j\colon\bs\to\bl$ by setting $j(x)\coloneqq\delta_x$. Then $(j,\bl)$ is a free $\infty$-convex Banach lattice with $\infty$-convexity constant 1 over $\pobs$.
\end{theorem}

\begin{remark}\label{rem:completeness_not_needed_infinity}\quad
	\begin{enumerate}
		\item We recall that  \cref{res:free_vector_lattice_over_pre-ordered_banach_space_as_lattice_of_functions} and \cref{res:FVL_embeds} show that the vector sublattice of $\bl$ that is generated by $\{\delta_x:x\in\bs\}$ is a free vector lattice over the pre-ordered vector space $\pobsclosure$.
		\item To further illustrate the difference between the present approach and that in \cite{aviles_rodriguez_tradacete:2018} and \cite{jardon-sanchez_laustsen_taylor_tradacete_troitsky:2022}, we note that  it is not necessary to observe in \cref{res:AM-realisation} that $\contballposh$ is complete in the maximum norm, in order to know that $\bl$ is complete. Being an isometric image of $\FBLinfpobs$, it {\emph is} complete. In the approach as in \cite{aviles_rodriguez_tradacete:2018} and \cite{jardon-sanchez_laustsen_taylor_tradacete_troitsky:2022} for $\bspos=\zerospace$, however, it is essential that the construction takes place in what is known from the outset to be a complete function lattice. Without this, there is no guarantee that the completion that is taken in the construction yields a function lattice again.
		\end{enumerate}
\end{remark}

\subsection{The case $1\leq p<\infty$} We now turn to the remaining case $1\leq p<\infty$, which is more demanding.\footnote{We also include the case $p=1$. This is excluded in \cite[Section~6]{jardon-sanchez_laustsen_taylor_tradacete_troitsky:2022}, perhaps because \cite{aviles_rodriguez_tradacete:2018} then already provides the answer with an easier argument, but for the proof as such this is not necessary.} Our aim is to show that, for $\freeelt\in\FBLppobs$,
\begin{equation}\label{eq:reduction_to_finite_dimensional_l_p}
\norm{\freeelt}=\sup_\mor\norm{\facmor(\freeelt)},
\end{equation}
where $\mor$ runs over the positive contractions from $\bs$ into $\ell_p^n$ for $n=1,2,\ldots$, and then proceed as in \cref{res:injectivity_of_canonical_map}.

The first step towards \eqref{eq:reduction_to_finite_dimensional_l_p} is to note that, for $\freeelt\in\FBLppobs$,
\begin{equation*}\label{eq:reduction_to_L_p}
\norm{\freeelt}=\sup_\mor\norm{\facmor(\freeelt)}
\end{equation*}
where $\mor$ runs over the positive contractions from $\bs$ into $\Ell_p(\mu_\mor)$-spaces for probability measures $\mu_\mor$. This is an immediate consequence of \cref{res:characterisation_all_p} and \cref{res:composition}.

We now start the next step, which is the passage from spaces $\Ell_p(\mu)$ as codomains to spaces $\ell_p^n$.

For $1\leq p\leq \infty$, we let $q$ denote its conjugate exponent. For $g\in L_q(\mu)$, where $\mu$ is not necessarily finite, we write $g^*$ for the corresponding element of $L_p(\mu)^*$. The following preparatory observation from \cite{jardon-sanchez_laustsen_taylor_tradacete_troitsky:2022} is at the same time elementary and crucial.

\begin{lemma}\label{res:trivial_and_crucial}
Let $1\leq p\leq \infty$, and let $f_1,\dotsc,f_k\in\Ell_p(\mu)$ for a \uppars{not necessarily finite} measure $\mu$ on a set $\pset$. Suppose that $S$ is a measurable subset of $\pset$ with the property that there exists $1\leq i^\prime\leq k$ such that $\bigvee_{i=i}^k f_i(\pt)=f_{i^\prime}(\pt)$ for all $\pt\in S$. If $g\in\Ell_q(\mu)$ is positive on $S$  and zero outside $S$, then
\[
g^*\left(\bigvee_{i=i}^k f_i\right)=\bigvee_{i=1}^k g^* (f_i).
\]
Similarly for the infimum.
\end{lemma}

We shall also use the following, which was implicitly already observed in \cite{jardon-sanchez_laustsen_taylor_tradacete_troitsky:2022}.

\begin{lemma}\label{res:disjoint_functions_yield_contraction}
	Let $1\leq p\leq\infty$.
	Suppose that $\Gamma=\{g_1,\dotsc,g_k\}\subset \Ell_q(\mu)$ for a \uppars{not necessarily finite} measure $\mu$ is such that $g_i g_j=0$ for $1\leq i<j\leq k$ and $\norm{g_i}_q\leq 1$ for $i=1,\dotsc,k$.
	For $f\in\Ell_p(\mu)$, set
	\[
	\psi_\Gamma(f)\coloneqq (g_1^*(f),\dotsc,g_k^*(f))\in\ell_p^k.
	\]
	Then $\psi_\Gamma\colon \Ell_p(\mu)\to\ell_p^k$ is a contraction. If all $g_i$ are positive, then $\psi_\Gamma$ is positive.
\end{lemma}

\begin{proof}
	Using that the unit ball of $\ell_q^k$ is norming for $\ell_p^k$ in the second equality, we have
	\begin{align*}
		\norm{\psi_\Gamma(f)}&=\left(\sum_{i=1}^k\abs{g_i^*(f)}^p\right)^{1/p}\\
		&=\max_{(a_1,\dotsc,a_k)\in\ball_{\ell_q^k}}\sum_{i=1}^k a_ig_i^*(f)\\
		&=\max_{(a_1,\dotsc,a_k)\in\ball_{\ell_q^k}}\left(\sum_{i=1}^k a_ig_i\right)^*(f).
	\end{align*}
	The disjointness of the $g_i$ is easily seen to imply that $\norm{\sum_{i=1}^ka_ig_i}_q\leq 1$ whenever $(a_1,\dotsc,a_k)\in\ball_{\ell_q^k}$. An application of the H\"older inequality concludes the proof.
\end{proof}

We now come to the proof proper for $1\leq p<\infty$. We use $\chi_A$ for the characteristic function of a set $A$.

Let $\freeelt$ be an element of the vector sublattice of $\FBLppobs$ that is generated by $j(\bs)$.  By \cite[Exercise~4.1.8]{aliprantis_burkinshaw_POSITIVE_OPERATORS_SPRINGER_REPRINT:2006},  there exist $m,n$ and $x_{kl}\in\bs$ for $1\leq k\leq m$ and $1\leq l\leq n$ such that
\[
\freeelt=\bigvee_{k=1}^m\bigwedge_{l=1}^n j(x_{kl}).
\]
Let $\mor\colon\bs\to\Ell_p(\mu)$ be a positive contraction for a (not necessarily finite) measure $\mu$ on a set $\pset$; we need only the case where $\mu$ is a probability measure, but the arguments below are valid in general. Then
\[
\facmor(\freeelt)=\bigvee_{k=1}^m\bigwedge_{l=1}^n \mor(x_{kl}).
\]
For $k^\prime=1,\dotsc,m$, set
\[
\widetilde A_{k^\prime}\coloneqq\left\{\pt\in\pset: \bigvee_{k=1}^m\bigwedge_{l=1}^n [\mor(x_{kl})](\pt)=\bigwedge_{l=1}^n [\mor(x_{k^\prime l})](\pt)\right\}.
\]
Then $\pset=\bigcup_{k^\prime=1}^m \widetilde A_{k^\prime}$. For $k^\prime=1,\dotsc,m$, set $A_{k^\prime}=\widetilde A_{k^\prime}\setminus\bigcup_{k<k^\prime}\widetilde A_{k}$, so that $\pset=\bigcup_{k^\prime=1}^m A_k^\prime$ as a disjoint union and
\[
A_{k^\prime}\subseteq\left\{\pt\in\pset: \bigvee_{k=1}^m\bigwedge_{l=1}^n [\mor(x_{kl})](\pt)=\bigwedge_{l=1}^n [\mor(x_{k^\prime l})](\pt)\right\}.
\]
Proceeding similarly for $\bigwedge_{l=1}^n \mor(x_{k^\prime l})$ on $A_{k^\prime}$, we obtain a disjoint union $A_{k^\prime}=\bigcup_{l^\prime=1}^n B_{k^\prime l^\prime}$ such that
\[
B_{k^\prime l^\prime}\subseteq\left\{\pt\in\pset: \bigwedge_{l=1}^n [\mor(x_{k^\prime l})](\pt)= [\mor(x_{k^\prime l^\prime})](\pt)\right\}.
\]
Finally, set
\[
B_{k^\prime l^\prime}^+=\left\{\pt\in B_{k^\prime l^\prime}: \bigvee_{k=1}^m\bigwedge_{l=1}^n [\mor(x_{kl})](\pt)\geq 0\right\}
\]
and
\[
B_{k^\prime l^\prime}^-=\left\{\pt\in B_{k^\prime l^\prime}: \bigvee_{k=1}^m\bigwedge_{l=1}^n [\mor(x_{kl})](\pt)< 0\right\},
\]
so that $\pset$ is the disjoint union of the $B_{k^\prime l^\prime}^+$ and $B_{k^\prime l^\prime}^-$. For all $k^\prime$ and $l^\prime$, choose a positive $g_{k^\prime l^\prime}^+\in\Ell_q(\mu)$ with $\norm{g_{k^\prime l^\prime}^+}\leq 1$ which is zero outside $B_{k^\prime l^\prime}^+$ and such that
\[
\left(g_{k^\prime l^\prime}^+\right)^*(\facmor(\freeelt))=\lrnorm{\chi_{B_{k^\prime l^\prime}^+}\facmor(\freeelt)}_{\Ell_p(\mu)},
\]
and a positive
$g_{k^\prime l^\prime}^-\in\Ell_q(\mu)$ with $\norm{g_{k^\prime l^\prime}^-}\leq 1$ which is zero outside $B_{k^\prime l^\prime}^-$ and such that
\[
\lrabs{\left(g_{k^\prime l^\prime}^-\right)^*(\facmor(\freeelt))}=\lrnorm{\chi_{B_{k^\prime l^\prime}^-}\facmor(\freeelt)}_{\Ell_p(\mu)}.
\]
From the disjointness of the union, we have
\begin{align}\label{eq:norm_splits}
\norm{\facmor(\freeelt)}_{\Ell_p(\mu)}^p&=\sum_{k^\prime=1}^m\sum_{l^\prime=1}^n \left(\lrnorm{\chi_{B_{k^\prime l^\prime}^+}\facmor(\freeelt)}_{\Ell_p(\mu)}^p + \lrnorm{\chi_{B_{k^\prime l^\prime}^-}\facmor(\freeelt)}_{\Ell_p(\mu)}^p\right)\notag\\
&=\sum_{k^\prime=1}^m\sum_{l^\prime=1}^n\left(\lrabs{\left(g_{k^\prime l^\prime}^+\right)^*(\facmor(\freeelt))}^p+\lrabs{\left(g_{k^\prime l^\prime}^-\right)^*(\facmor(\freeelt))}^p\right).
\end{align}
The point of all this is that, by a double application of \cref{res:trivial_and_crucial} in the second equality,
\begin{align*}
\left(g_{k^\prime l^\prime}^\pm\right)^*(\facmor(\freeelt))&=\left(g_{k^\prime l^\prime}^\pm\right)^*\left(\bigvee_{k=1}^m\bigwedge_{l=1}^n \mor(x_{kl})\right)\\
&= \bigvee_{k=1}^m\bigwedge_{l=1}^n \left(g_{k^\prime l^\prime}^\pm\right)^* [\mor(x_{kl})]\\
&= \bigvee_{k=1}^m\bigwedge_{l=1}^n \left(\left(g_{k^\prime l^\prime}^\pm\right)^*\circ\mor\right) (x_{kl})\\
&= \bigvee_{k=1}^m\bigwedge_{l=1}^n \fac{\left(\left(g_{k^\prime l^\prime}^\pm\right)^*\circ\mor\right)} (j(x_{kl}))\\
&= \fac{\left(\left(g_{k^\prime l^\prime}^\pm\right)^*\circ\mor\right)} \left(\bigvee_{k=1}^m\bigwedge_{l=1}^n j(x_{kl})\right)\\
&=\fac{\left(\left(g_{k^\prime l^\prime}^\pm\right)^*\circ\mor\right)}(\freeelt).
\end{align*}
Thus \eqref{eq:norm_splits} gives that
\begin{equation}\label{eq:interpretable_expression}
\norm{\facmor(\freeelt)}_{\Ell_p(\mu)}^p=\sum_{k^\prime=1}^m\sum_{l^\prime=1}^n\left(\lrabs{\fac{\left(\left(g_{k^\prime l^\prime}^+\right)^*\circ\mor\right)}(\freeelt)}^p+\lrabs{\fac{\left(\left(g_{k^\prime l^\prime}^-\right)^*\circ\mor\right)}(\freeelt)}^p\right).
\end{equation}
We shall now interpret this. In view of \cref{res:disjoint_functions_yield_contraction}, the set
\[
\Gamma\coloneqq\left\{\left(g_{k^\prime l^\prime}^\pm\right)^*:1\leq k^\prime\leq m,1\leq l^\prime\leq n\right\}
\]
yields a positive contraction $\psi_\Gamma$ from $\Ell_p(\mu)$ into $\ell_p^{2mn}$. Then $\psi_\Gamma\circ\mor\colon\bs\to\ell_p^{2mn}$ is a positive contraction, the coordinate components of which are the  $\left(g_{k^\prime l^\prime}^\pm\right)^*\circ\mor$. By the uniqueness of the factor morphism, the coordinate components of $\fac{\psi_\Gamma\circ\mor}$ are the $\fac{\left(g_{k^\prime l^\prime}^\pm\right)^*\circ\mor}$. Thus \eqref{eq:interpretable_expression} states that \[
\norm{\facmor(\freeelt)}_{\Ell_p(\mu)}^p=\norm{\fac{\left(\psi_\Gamma\circ\mor\right)}(\freeelt)}_{\ell_p^{2mn}}^p.
\]

Combining this with \eqref{eq:reduction_to_L_p}, we have established the following.\footnote{For elements $\freeelt$ of the form $\bigvee_{k=1}^m j(x_j)-\bigvee_{l=1}^n j(y_l)$ as in \cite{jardon-sanchez_laustsen_taylor_tradacete_troitsky:2022}, the norm is obtained from the positive contractions into $\ell_p^{2mn}$.}

\begin{proposition}\label{res:norm_on_FVL_specifici}
For $\freeelt=\bigvee_{k=1}^m\bigwedge_{l=1}^n j(x_{kl})$ in $\FBLppobs$, its norm is given by
\[
\norm{\freeelt}=\sup_{\mor}\norm{\facmor(\freeelt)},
\]
where $\mor$ runs over the positive contractions from $\bs$ into $\ell_p^{2mn}$.
\end{proposition}

Now the work has been done. Evidently, if $\freeelt$ is any element of the vector sublattice of $\FBLppobs$ that is generated by $j(\bs)$, then
\[
\norm{\freeelt}=\sup_{\mor}\norm{\facmor(\freeelt)}
\]
as $\mor$ runs over the positive contractions from $\bs$ into $\ell_p^n$ for $n\geq 1$. On invoking \cref{res:norming}, we have the following.

\begin{theorem}\label{res:norm_p-convex}
Let $1\leq p<\infty$, and let $\pobs$ be a pre-ordered Banach space. For $\freeelt\in\FBLppobs$, its norm is given by
\begin{equation}\label{eq:norm_p-convex}	
\norm{\freeelt}=\sup_{\mor}\norm{\facmor(\freeelt)},
\end{equation}
where $\mor$ runs over the positive contractions from $\bs$ into $\ell_p^n$ for $n\geq 1$.	
\end{theorem}

Using \eqref{eq:norm_p-convex}, we shall now realise $\FBLppobs$ as a function lattice on $\ballbsdualpos$ as in \cref{res:injectivity_of_canonical_map}.

The positive contractions $\mor\colon\bs\to\ell_p^n$ correspond to subsets $\{x_1^*,\dotsc,x_n^*\}$ of $\ballbsdualpos$ such that
\[
\sup_{x\in\ballbs}\sum_{i=1}^n\abs{x_i^*(x)}^p\leq 1,
\]
in which case the $\fac\mor$ are easily seen to be given by $\{\fac{x_1^*},\dotsc,x\fac{_n^*}\}$, as a consequence of the uniqueness of a factoring morphism. Hence

\begin{equation*}\label{eq:norm_formula_p}
\norm{\freeelt}\!=\!\sup\left\{\!\left(\sum_{i=1}^n\lrabs{\fac{ x_i^*}(\freeelt)}^p\!\right)^{1/p}\!\!\!\!\!\!: n\geq 1, x_1^*,\dotsc,x_n^*\in\ballbsdualpos \! \text{ s.t.\ }\!\!\sup_{x\in\ballbs}\!\sum_{i=1}^n\abs{x_i^*(x)}^p\leq 1 \!\right\}\!.
\end{equation*}
 For $f\in\contballposh$, define $\norm{f}_p\in[0,\infty]$ by setting
\[
\norm{f}_p\!\!\coloneqq\sup\!\left\{\!\!\left(\sum_{i=1}^n\lrabs{f(x_i^*)}^p\!\!\right)^{1/p}\!\!\!\!\!\!\!\!\!\!: n\geq 1, x_1^*,\dotsc,x_n^*\in\ballbsdualpos\!\!  \text{ s.t.\ }\!\!\sup_{x\in\ballbs}\!\sum_{i=1}^n\abs{x_i^*(x)}^p\!\leq \!1\! \right\}\!.
\]
Set
\[
\contballposhp\coloneqq\left\{f\in\contballposh:\norm{f}_p<\infty\right\}.
\]
This is a vector lattice, and it is routine to verify that it is a Banach lattice when supplied with the vector lattice norm $\norm{\,\cdot\,}_p$.

For $\freeelt\in\FBLppobs$, define $\widehat\freeelt\colon\ballbsdualpos\to\RR$ by setting
\[
\widehat \freeelt (x^*)\coloneqq \fac{x^*}(\freeelt).
\]
Then $\widehat{j(x)}=\delta_x$ for $x\in\bs$. Recalling that $\FBLinfpobs$ is generated by $j(\bs)$, \cref{res:norm_p-convex} yields the following.

\begin{theorem}\label{res:p-realisation}
	Let $\pobs$ be a pre-ordered Banach space. Let $E$ be the closed vector sublattice of the Banach lattice $\contballposhp$ that is generated by $\{\delta_x:x\in\bs\}$. Define $j\colon\bs\to\bl$ by setting $j(x)\coloneqq\delta_x$. Then $(j,\bl)$ is a free $p$-convex Banach lattice with $p$-convexity constant 1 over $\pobs$.
\end{theorem}

\begin{remark}\label{rem:completeness_not_needed_p}Analogously to \cref{rem:completeness_not_needed_p}, the vector sublattice of $\bl$ that is generated by $\{\delta_x:x\in\bs\}$ is a free vector lattice over the pre-ordered vector space $\pobsclosure$, and the completeness of $\contballposhp$ is not needed to know that $\bl$ is complete.
\end{remark}


\subsection*{Acknowledgements}
The authors are indebted to Jochen Gl\"{u}ck for pointing out the relevance of the normality of the positive cone in \cref{sec:free_banach_lattices_over_pre-ordered_banach_spaces}, and to the anonymous referees for their careful reading of the manuscript. This research was supported by the Erasmus+ ICM programme, the National Natural Science Foundation of China (Grant No.\ 12201439), and the Natural Science Foundation of Sichuan Province (Grant No.\ 2024NSFSC1339).


\bibliographystyle{plain}
\urlstyle{same}

\bibliography{general_bibliography}



\end{document}